\documentclass[
framed_theorems,
framed_proofs,
print,
]{OCPreprint}

\usepackage[ruled,vlined,linesnumbered,algo2e,algosection]{algorithm2e}

\usepackage{etoolbox}
\makeatletter
\patchcmd{\@algocf@start}{\addtolength{\hsize}{-1.5em}}{}{}{}
\makeatother

\usepackage{pgfplots}
\usepgfplotslibrary{fillbetween}
\usepackage{pgfplotstable}
\pgfplotsset{compat=1.16}

\usepackage{xcolor}
\usepackage{subcaption}

\usepackage{enumitem}
\setlist[enumerate,1]{label=\textup{(\roman*)}}

\AtEndPreamble{
	
	\ifthenelse{\boolean{framedtheorems}}{
		\maketheoremcolored{notation}
		\colorlet{notationframecolor}{assumptionframecolor}
		\colorlet{notationshadecolor}{assumptionshadecolor}
	}{}
}

\newcommand\TV{\operatorname{TV}}
\newcommand\TVp{\operatorname{TV}_p}
\newcommand\pVp{\operatorname{pV}_p}
\newcommand\eVp{\operatorname{eV}_p}

\newcommand\BV{\operatorname{BV}}
\newcommand\F{\mathbb{F}}

\newcommand\ared{\operatorname{ared}}
\newcommand\pred{\operatorname{pred}}

\newcommand\STfull{\hyperref[eq:ST]{\textup{(ST($\hat n, \hat a$))}}}
\newcommand\STbar{\hyperref[eq:ST]{\textup{(ST($\bar n, \bar a$))}}}

\definecolor{BTUorange}{RGB}{255,128,0}%orange
\definecolor{BTUdarkblue}{RGB}{0,0,250}%dunkelblau
\definecolor{BTUred}{RGB}{255,40,20}%rot
\definecolor{BTUlightblue}{RGB}{51,153,255}%hellblau
\definecolor{BTUgreen}{RGB}{0,153,0}%grün

\usepackage{pgfplotstable}
\usepgfplotslibrary{fillbetween}
\pgfplotsset{
	extra x tick style={color = cyan}
}

\title[Vector-Valued Integer Optimal Control with TV Regularization]%
{Vector-Valued Integer Optimal Control with TV Regularization: Optimality Conditions and Algorithmic Treatment}
\author[Marko, Wachsmuth]{%
	Jonas Marko%
	\footnote{%
		Brandenburgische Technische Universität Cottbus--Senftenberg,
		Institute of Mathematics,
		03046 Cott\-bus,
		Germany,
		\email{markojo1@b-tu.de},
		\url{https://www.b-tu.de/fg-optimale-steuerung},
		ORCID: \href{https://orcid.org/0000-0001-7745-1254}{0000-0001-7745-1254}%
	},
	Gerd Wachsmuth%
	\footnote{%
		Brandenburgische Technische Universität Cottbus--Senftenberg,
		Institute of Mathematics,
		03046 Cott\-bus,
		Germany,
		\email{gerd.wachsmuth@b-tu.de},
		\url{https://www.b-tu.de/fg-optimale-steuerung},
		ORCID: \href{https://orcid.org/0000-0002-3098-1503}{0000-0002-3098-1503}%
	}
}

\begin{document}
	%%fakesection: Title, abstract und co
	
	% \date{\today}
	% \publishers{}
	% \dedication{}
	\maketitle
	
	\begin{abstract}
		We investigate a broad class of integer optimal control problems with vector-valued controls and switching regularization using a total variation functional involving the $p$-norm, which influences the structure of a solution. 
		We derive optimality conditions of first and second order for the integer optimal control problem via a switching-point reformulation. 
		
		For the numerical solution, we use a trust region method utilizing Bellman's optimality principle for the subproblems. We will show convergence properties of the method and highlight the algorithms efficacy on some benchmark examples. 
	\end{abstract}
	
	
	\begin{msc}
		\mscLink{49K30},
		\mscLink{49L20}
		\mscLink{49M37},
		\mscLink{90C10}
	\end{msc}
	
	%%fakesection: Introduction
	\section{Introduction}
	\label{sec:intro}
	Integer optimal control problems (IOCPs) offer a framework to address decision-making scenarios where discrete choices govern continuous processes, thus, leading to applications in various domains, such as in automotive control (\cite{Gerdts2005,Kirches2010}), energy transmission (\cite{Goettlich2018}), robotics (\cite{ZhangLiGuo2014}), or in controlling gas networks (\cite{RuefflerHante2016,HabeckPfetschUlbrich2019}), traffic (\cite{Goettlich2014, Goettlich2017}) and the heating of buildings (\cite{Burger2018,Zavala2010}). 
	
	In tackling IOCPs, the combinatorial integral approximation has emerged as a popular methodological approach, see e.g. \cite{SagerJungKirches2011,ZeileWeberSager2022}. However, the computation of the binary controls in the rounding step is decoupled from the dynamics and objective of the IOCP, which may lead to poor approximation performance, as highlighted in \cite[p.~139]{Buerger2021}. Furthermore, it produces control functions with chattering behaviour, which is often undesired.
	Addressing this concern, total variation regularization has gained prominence as an effective mechanism resulting in optimal solutions without chattering. While recent research has explored IOCPs with multidimensional domains (\cite{MannsSchiemann2023,SeverittManns2024}), our focus is on the time-dependent control setting.

	Particularly, we are interested in IOCPs with vector-valued controls.
	For the case of IOCPs governed by ordinary differential equations (ODEs), \cite{BestehornHansknechtKirchesManns2020} suggests a switching cost aware rounding method for the solution of a relaxed problem.
	Here, the relaxed problem ignores the switching costs
	and, consequently,
	the switching costs of the rounded solution
	converge towards $\infty$ as the time discretization goes to $0$.
	In the recent contributions \cite{SagerZeile2021,SagerTetschkeZeile2024},
	the switching cost is included in the relaxed problem,
	however, this is not the correct relaxation of the original problem,
	see \cite[Counterexample~3.1]{BuchheimGrueteringMeyer2024:1}.

	The case of IOCPs governed by a partial differential equation (PDE) was investigated in \cite{HahnKirchesMannsSagerZeile2022}, where the approach of the combinatorial integral approximation was generalized. In \cite{BuchheimGrueteringMeyer2024:1,BuchheimGrueteringMeyer2024:2}, the convex hull of the set of feasible controls is constructed as intersection of convex sets derived from finite-dimensional projections, and an algorithm to solve this relaxed problem is presented.

	\subsection*{Contribution}
	We present an analysis of a broad class of IOCPs with time-dependent, vector-valued controls. We use a generalized total variation functional to stimulate the appearance of simultaneous switches between the control components. The properties of this functional will be investigated in \cref{sec:TVp}. We generalize the approach of \cite{LeyfferManns2021,MarkoWachsmuth2022:1} to get optimality conditions of first and second order via a switching-point optimization problem, see \cref{sec:conditions}. The main ingredient is \cref{thm:equivalence}, which shows equivalence of local optimality (or a local quadratic growth condition) w.r.t.\ the original problem and the switching-point reformulation. For this, admissible controls need to be encoded following a special construction, see \cref{def:full_representation}. Furthermore, we transfer the obtained conditions regarding the switching-point problem into a more accessible form in \cref{thm:opt_conditions}.
	
	In the second part of the paper, the focus lies in the analysis of a trust-region algorithm in \cref{sec:convergence} and a solution method for the arising subproblems involving Bellman's optimality principle in \cref{sec:bellman}. Finally, we present numerical experiments for two IOCPs, one governed by an ODE, the other by a PDE in \cref{sec:examples}.
	
	\subsection*{Notation}
	We write $\norm{v}_p$ for the usual $p$-norm of a vector $v\in\R^M$ for $p\in[1,\infty]$. 
	If $u\in L^p(0,T)^M$, we set $\norm{u}_{L^p}:=\left(\int_0^T \norm{u(t)}_p^p\,\d t\right)^{\frac 1p}$ if $p\in[1,\infty)$ 
	and $\norm{u}_{L^\infty}:= \esssup_{t\in(0,T)} \norm{u(t)}_{\infty}$. We define $[d] := \{1,\dots,d\}$ and $[d]_0:=\{0,\dots,d\}$ for any $d\in\N_0:=\N\cup\{0\}$. Note that $[d] = \emptyset$ if $d=0$.
	
	\section{Problem Formulation and Properties}
	\label{sec:problem}
	We investigate problems of the form
	\begin{equation}
		\label{eq:prob}
		\tag{P}
		\begin{aligned}
			\text{Minimize} \quad & F(u) + \beta \TVp(u) \\
			\text{such that} \quad & u(t) \in 
			\VV\subset \Z^M\quad\text{ for a.a.\ } t \in (0,T).
		\end{aligned}
	\end{equation}
	Here, the admissible control values are integers in every component and 
	the set $\VV$ is assumed to be finite. Thus, we will denote it by
	$\VV:=\set{ \nu_1, \ldots, \nu_d}$,
	where $d\in\N$, $d\geq 2$ and $\nu_i \in\Z^M$ for all $i\in[d]$. 
	
	In the objective, the functional $F \colon L^1(0,T)^M \to \R$ is assumed to be bounded from below, lower semicontinuous and Gâteaux differentiable. It might contain the solution operator of a differential equation, i.e. $F(u) = \hat F(S(u),u)$, where $S$ is a control-to-state operator, e.g., mapping $u$ to a state satisfying some ODE or PDE, see also \cref{sec:examples}.
	
	In order to show the existence of a solution and to impose a nice structure on an optimal control, 
	the regularization term $\TVp$ weighted by $\beta >0$ is used. 
	The functional $\TVp$ works, in essence, in the same way as the standard total variation by penalizing jumps of the control. 
	In order to gain more influence on the switching structure of the (componentwise) piecewise constant function $u$, we will utilize the $p$-norm for $p\in[1,\infty]$ and define $\TVp$ by 

	\begin{equation*}
		\TVp(u,J) := \sup\set*{\int_J u^\top \varphi' \,\d t\given \varphi\in C_c^1(J)^M,\,\norm{\varphi(t)}_{p'}\leq 1 \,\forall t\in J},
	\end{equation*}
	where $p'$ is the Hölder conjugate of $p$ and $J\subset(0,T)$ is open. Also, we write $\TVp(u):=\TVp(u,(0,T))$. If $u$ is differentiable, it can be shown that 
	$\TVp(u) = \int_0^T\norm{u'(t)}_p\,\d t$. 
	 
	In \cref{ex:p-vals}, we motivate how different values of $p$ may influence the behaviour of an optimal control.
	\begin{example}
	\label{ex:p-vals}
	For $M=2$, consider $u=(u_1,u_2)^\top$ with $u_1, u_2$ as depicted in \cref{fig:ex_p-vals},
	i.e.,
		\[u_1(t) = \begin{cases}
				0 & \text{if } t\leq 2\\
				2 & \text{if } t > 2
		\end{cases}
		\quad\text{and}\quad
		u_2(t) = \begin{cases}
			2 & \text{if } t\leq 2+\varepsilon\\
			0 & \text{if } t > 2+\varepsilon
	\end{cases}\]
	with $\varepsilon \in [0,2)$.
	We have
	\[\begin{aligned}
	\TV_1(u)
	%=\abs{2-0}+ \abs{1-(-1)}
	= 4,\quad 
	\TV_2(u) 
	%= \sqrt{\abs{2-0}^2+ \abs{1-(-1)}^2 = \sqrt 8} 
	&= \begin{cases}
		2\sqrt 2 & \text{if } \varepsilon = 0\\
		4 & \text{if } \varepsilon>0,
	\end{cases}\quad
	\TV_4(u)
	%= \sqrt[4]{\abs{2-0}^4 + \abs{1-(-1)}^4} = \sqrt[4]{32} 
	= \begin{cases}
		2 \sqrt[4]{2} & \text{if } \varepsilon = 0\\
		4 & \text{if }\varepsilon>0,
	\end{cases}\\
	\TV_\infty(u)
	%= \max(\abs{2-0}, \abs{1-(-1)})
	&= \begin{cases}
		2 & \text{if } \varepsilon = 0\\
		4 & \text{if } \varepsilon>0.
	\end{cases}
\end{aligned}\]
	For increasing $p$, any shared discontinuities of $u_1$ and $u_2$ will incur a smaller penalization by $\TVp$, thus encouraging shared switches in the components of $u$. 
	
	For admissible $u$, the following is true:
	\begin{itemize}
		\item $\TV_1(u) = \sum_{i=1}^{M}\TV(u_i)$
		\item $\TV_2(u) = \TV(u)$, where $\TV$ is the classic total variation functional, see \eqref{eq:TV} below.
		\item $\TV_\infty$ penalizes shared discontinuities only once by the value of the biggest jump
	\end{itemize} 
	\end{example}
	
	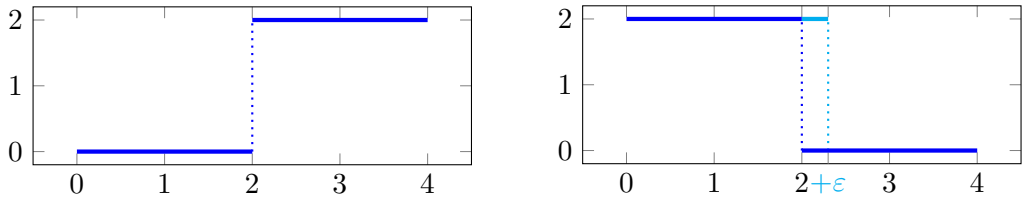
\begin{figure}[htp]
		\centering
		\begin{tikzpicture}
			\begin{axis}[width=.5\textwidth,height = .25\textwidth,xmin=-0.5,xmax=4.5, xtick ={0,1,2,3,4}]
				% Horizontal lines
				\addplot[BTUdarkblue, ultra thick, mark={}] coordinates{
					(0,0) (2,0)
					
					(2,2) (4,2)
				};
				
				\addplot[blue, thick, dotted, mark={}] coordinates{
					(2,0) (2,2)
				};
			\end{axis}
		\end{tikzpicture}%
		\hspace{1cm}% 
		\begin{tikzpicture}%
			\begin{axis}[width=.5\textwidth,height = .25\textwidth, xmin=-0.5,xmax=4.5, xtick ={0,1,2,3,4},
				xticklabels = {0,1,2,3,4},
				extra x ticks={2.3},
				%extra x tick style={color = cyan},
				extra x tick labels={\color{cyan}$+\varepsilon$}]
				% Horizontal lines
				\addplot[cyan, thick, dotted, mark={}] coordinates{
					(2.3,0) (2.3,2)
				};
				\addplot[BTUdarkblue, ultra thick, mark={}] coordinates{
					(0,2) (2,2)
					
					(2,0) (4,0)
				};
			\addplot[cyan, ultra thick, mark={}] coordinates{
				(2,2) (2.3,2)
			};
				
				\addplot[BTUdarkblue, thick, dotted, mark={}] coordinates{
					(2,0) (2,2)
				};
			
			\end{axis}
		\end{tikzpicture}%
		\caption{%
			The components of $u\in L^1(0,T)^2$ with $u(t)\in \{0,1,2\}^2$.
		}
		\label{fig:ex_p-vals}
	\end{figure}
	Finally, the set of admissible controls is denoted by
	\[\Uad := \set*{u\in\BV(0,T)^M\given u(t)\in\VV\,\text{ for a.a.\ } t \in (0,T)},\]
	where $\BV(0,T)^M$ is the space of $L^1(0,T)^M$-functions with bounded variation, equipped with the norm 
	\[\norm u_{\BV(0,T)^M} = \norm u_{L^1} + \TV_2(u),\]
	see also \cite[Section 3.1]{Ambrosio2000}.
	
	\section{The $\TVp$ Functional}
	\label{sec:TVp}
	In this section, we want to show that the $\TVp$ functional works, in essence, the same way as the well-known $\TV$ functional defined by
	
	\begin{equation}
		\label{eq:TV}
		\TV(u):=\sup\set*{\int_0^T u^\top \varphi' \, \d t\given \varphi\in C^1_c(0,T),~\norm{\varphi(t)}_{2} \le 1\ \forall t\in(0,T)},
	\end{equation}
	c.f.\ \cite[Definition 3.4]{Ambrosio2000}.
	First, we need some general results on $\BV(0,T)^M$ and $\TVp$.

	\begin{theorem} 
		\label{thm:props}
		The space $\BV(0,T)^M$ and the functional $\TVp$ have the following properties.
		\begin{enumerate}
			\item\label{thm:props:0}
			The space $\BV(0,T)^M$ is (isometric isomorphic to) the dual space of a separable Banach space.
			\item\label{thm:props:1} For a sequence $\seq{u_k}_{k\in\N}\subset \BV(0,T)^M$, we have $u_k\weaklystar 
			u$ in $\BV(0,T)^M$ if and only if $u_k\to u$ in $L^1(0,T)^M$ and 
			$\seq{u_k}_{k\in\N}$ is bounded in $\BV(0,T)^M$.
			\item\label{thm:props:2} $\BV(0,T)^M$ is continuously embedded in $L^\infty(0,T)^M$ and compactly embedded in $L^q(0,T)^M$ for all $q\in[1,\infty)$.
			\item\label{thm:props:3} When $u_k \weaklystar u$ in $\BV(0,T)^M$, we have $u_k\to u$ in $L^q(0,T)^M$ for all $q\in[1,\infty)$.
			\item\label{thm:props:4} If $\seq{u_k}_{k\in\N}$ is bounded in $\BV(0,T)^M$, there exists a weak-$\star$ accumulation point of $\seq{u_k}$.
			\item\label{thm:props:5}
			The functional $\TVp$ is lower semicontinuous on $L^1(0,T)^M$,
			i.e., $u_k \to u$ in $L^1(0,T)^M$ implies $\TVp(u) \le \liminf_{k \to \infty} \TVp(u_k)$.
		\end{enumerate}
	\end{theorem}
	\begin{proof}
		For \ref{thm:props:0} and \ref{thm:props:1}, \ref{thm:props:2} 
		%and \ref{thm:props:5}, 
		see \cite[Remark 3.12, Proposition 3.13, Corollary 3.49
		% and Remark 3.5
		]{Ambrosio2000}.
		Assertion \ref{thm:props:5} can be shown the same way as in \cite[Remark 3.5]{Ambrosio2000} by noting that $u\mapsto \int_0^T u^\top \varphi'\,\d t$ is continuous on $L^1(0,T)^M$. 
		Assertion \ref{thm:props:4} is a direct consequence of \ref{thm:props:0}. 
		
		Finally, we show \ref{thm:props:3} by noting that an interpolation inequality yields that
		\[
			\norm{u_k-u}_{L^q} \leq\norm{u_k-u}_{L^1}^{1/q}\norm{u_k-u}_{L^\infty}^{1-1/q}\to 0,
		\]
		where $\norm{u_k-u}_{L^1}\to 0$ and the boundedness of $\seq{u_k-u}_{k\in\N}$ in $\BV(0,T)^M$ (and thus, in $L^\infty(0,T)^M$ by \ref{thm:props:2}) follow from 
		\ref{thm:props:1}.

	\end{proof}

	With \cref{thm:props}, the existence of a solution of \eqref{eq:prob} can be shown by standard arguments: 
	A minimizing sequence $(u_k)_{k\in\N}\subset \Uad$ is bounded in $L^1(0,T)^M$ by $T\max_{i\in[d]}\set{\norm{\nu_i}_1}$, 
	% IF $\norm{u}_{L^p(0,T)^M}^p:= \int \norm{u(t)}_p^p\,\d t$. If we use Ambrosio Fusco Pallara norm, change this!
	while the boundedness of $\TVp(u_k)$ follows from the existence of a lower bound of $F$.
	Using \itemref{thm:props:4}, 
	the existence of a weak-$\star$ convergent subsequence $(u_{k_l})_{l\in\N}$
	with $u_{k_l}\weaklystar \bar u\in\BV(0,T)^M$ can be derived.
	Considering \itemref{thm:props:1}, 
	we see that $u_{k_l}\to\bar u$ in $L^1(0,T)^M$. Thus, 
	there is another subsequence converging a.e. pointwise to $\bar u$. 
	It follows that $\bar u(t)\in \VV$,
	hence $\bar u\in\Uad$. 
	Finally, the lower semicontinuity of $F$ and $\TVp$ yield the optimality of $\bar u$.
	
	The $\TV$ functional can be written using the distributional derivative of $u$. To adapt this for the functional $\TVp$,
	we recall the total variation 
	\[\abs{\mu}_p(E):=\sup\set*{\sum_{j=0}^{\infty}\norm{\mu(E_j)}_p\given E_j\in\mathcal{E}\text{ pairwise disjoint}, E=\bigcup_{j=0}^\infty E_j}\quad\forall E\in\EE\]
	of a vector-valued measure $\mu\colon\EE\to\R^M$ on a measurable space $(X,\mathcal{E})$ w.r.t.\ the $p$-norm.
	\begin{theorem}
		\label{thm:Du}
 		Let $Du$ be the finite Radon measure representing the distributional derivative of $u\in\BV(0,T)^M$. Then, $\TVp(u,J) = \abs{Du}_p(J)$ for any open set $J\subset (0,T)$.
	\end{theorem}
	\begin{proof}
		The result was shown for $p=2$ in \cite[Proposition 3.6]{Ambrosio2000}. It can be checked that only small changes in the proofs and definitions of \cite[Chapter 1]{Ambrosio2000} are necessary in order to get the result for $p\in[1,\infty]$.
	\end{proof}
	
	In \cite[Chapter 3.2]{Ambrosio2000}, a pointwise variation for some function $u\colon(0,T)\to\R^M$ was introduced, such that its value coincides with $\TV(v)$, where $v=u$ a.e. in $(0,T)$. We want to show a similar result. Thus, we first define a pointwise variation utilizing the $p$-norm.
	
	\begin{definition}
		\label{def:pVp_eVp}
		For $u\colon(0,T)\to\R^M$, the pointwise $p$-variation $\pVp(u)$ of $u$ is defined by
		\begin{equation*}
			\pVp(u):= \sup\set*{\sum_{i=1}^{n-1}\norm{u(t_{i+1})-u(t_i)}_p\given n\geq 2,\, 0<t_1<\ldots<t_n < T}.
		\end{equation*}	
		
		Furthermore, the essential $p$-Variation $\eVp(u)$ is defined by
		\begin{equation}
			\label{eq:eVp}
			\eVp(u) := \inf\set*{\pVp(v)\given v=u\text{ a.e. in }(0,T)}.
		\end{equation}
	
	\end{definition}
	
	\begin{theorem}
		\label{thm:eVp_TVp}
		For $u\in L^1(0,T)$, the infimum in \eqref{eq:eVp} is achieved and $\TVp(u)$ coincides with $\eVp(u)$.
	\end{theorem}
	\begin{proof}
		Similarly to \cite[Theorem 3.27]{Ambrosio2000}, we first prove $\TVp(u)\leq \eVp(u)$, i.e. that $\TVp(u)\leq \pVp(v)$ for any $v$ in the equivalence class of $u$ and assume w.l.o.g. that $\pVp(v)<\infty$. 
		For any integer $j\geq 1$ and the collection of points $\{t_i^j\}_{i\in [n_j]_0}$ with 
		\[\begin{aligned}
			&t_{0}^j = 0,\quad t_{n_j}^j = T,\quad 0<t_{i+1}^j-t_i^j\leq \frac 1j\quad\forall i\in[n_j-1],
		\end{aligned}\]
		we define the step functions 
		\[v_j(t):= \sum_{i=0}^{n_j-1}v(y_i^j)\chi_{(t_i^j,t_{i+1}^j]}(t)\]
		for some fixed $y_i^j\in (t_i^j,t_{i+1}^j]$. 
		With this definition, $v_j\to v=u$ in $L^1(0,T)^M$ as $j \to\infty$. 
		
		Now, we can calculate $\TVp(v_j)$ by
		\begin{equation*}
		\label{eq:TVp_step_calc}
			\begin{aligned}
				\TVp(v_j)
				&= \sup\set*{\sum_{i=0}^{n_j-1}v(y_i^j)^\top\parens*{\varphi(t_{i+1}^j)-\varphi(t_i^j)}\given
					\begin{aligned}
						&\varphi\in C_c^1(0,T)^M,\\
						&\norm{\varphi(t)}_{p'}\leq 1 \,\forall t\in(0,T)
					\end{aligned}
				}\\
				&= \sup\set*{\sum_{i=1}^{n_j-1} \varphi(t_i^j)^\top\parens*{v(y_{i-1}^j)-v(y_i^j)}\given
					\begin{aligned}
						&\varphi\in C_c^1(0,T)^M,\\
						&\norm{\varphi(t)}_{p'}\leq 1 \,\forall t\in(0,T)
					\end{aligned}
				}\\
				&= \sum_{i=1}^{n_j-1}\norm*{v(y_i^j)-v(y_{i-1}^j)}_p= \pVp(v_j)\leq \pVp(v).
			\end{aligned}
		\end{equation*}
	Using the lower semicontinuity of $\TVp$, we get $\TVp(u)\leq \pVp(v)$. 
	
	To show $\TVp(u)\geq \eVp(u)$, it is easy to apply the same arguments as in \cite{Ambrosio2000} using \cref{thm:Du}.
	\end{proof}
	
	\cref{thm:eVp_TVp} allows us to work with the more intuitive term $\eVp$, such that we can e.g. easily prove the following:
	\begin{corollary}
		\label{cor:TV_q_p}
		For all $p,q\in[1,\infty]$ and $u\in L^1(0,T)$, there are constants $c,C>0$ such that $c\TV_q(u)\leq \TVp(u)\leq C\TV_q(u)$. 
	\end{corollary}
 \begin{proof}
 	From the equivalence of all norms in $\R^n$, we get the existence of $c,C>0$ such that $c\,\text{pV}_q(v)\leq \pVp(v)\leq C\,\text{pV}_q(v)$ for any $v$ in the equivalence class of $u$. The statement follows from \cref{thm:eVp_TVp}.
 \end{proof}
	\section{Optimality Conditions}
	\label{sec:conditions}
	In order to investigate optimality conditions of \eqref{eq:prob}, 
	we want to adopt the approach from \cite[Chapter 3]{MarkoWachsmuth2022:1}. 
	First, we note that every $u \in \Uad$
	can uniquely be written as
	\begin{equation}
		\label{eq:min_representation}
		u=v^{t,a}:= a_1 \chi_{(t_0, t_1)}+\sum_{j = 2}^n a_j \chi_{[t_{j-1}, t_{j})}
	\end{equation}
	for $n \in \N$,
	$t\in\R^{n-1}$ with $t_i< t_{i+1}$ for all $i\in[n-1]_0$, $t_0 = 0$, $t_n = T$
	and $a=(a_1,\dots,a_n) \in \VV^n$ with $a_i \ne a_{i+1}$ for all $i \in [n-1]$.
	This follows from a slight generalization of
	\cite[Proposition 4.4]{LeyfferManns2021}. We call $v^{t,a}$ with $t$, $a$ defined as above the minimal representation of $u$.

	Next, we want to generalize the full representation (see \cite[Notation 3.3]{MarkoWachsmuth2022:1})
	to the vector-valued case.
	The disadvantage of the minimal representation is
	that admissible controls $w$ in a neighborhood of $u$ with $\TVp(w) = \TVp(u)$
	might not be representable as $w = v^{s, a}$
	with some $s \in \R^{n-1}$
	if we use the data $(n, t, a)$
	from the minimal representation.
	This is illustrated in the next example.
	\begin{example}
		\label{ex:Lambda}
		Let $M=2$, $\VV:=\set{0,1,2}^2$,
		see also \cref{fig:switching_paths}.
		Consider the minimal representation of some $u\in\Uad$, where $a_i = (0,0)^\top$ and $a_{i+1} = (2,2)^\top$ for some $i\in[n-1]$. Then, we have $\norm*{a_{i+1}-a_i}_p = \norm*{a_{i+1}-(1,1)^\top}_p + \norm*{(1,1)^\top-a_i}_p$ for all $p\in[1,\infty]$, 
		i.e., switching first to $(1,1)^\top$ from $a_i$ does not incur any additional cost in $\TVp(u)$. 
		This is possible since $(1,1)^\top$
		is a convex combination of
		$(0,0)^\top$ and $(2,2)^\top$.
		Hence, admissible controls $w$ in the neighborhood of $u$
		might switch from $a_i = (0,0)^\top$
		to $(1,1)^\top$ before going to $a_{i+1} = (2,2)^\top$.
		This is similar to \cite[Example~3.2]{MarkoWachsmuth2022:1}.

		For
		$p\in\{1,\infty\}$, 
		the situation is more complicated,
		since the $p$-norm is no longer strictly convex on $\R^M$.
		Consider the case that
		$a_i = (0, 0)^\top$
		and
		$a_{i+1} = (2,1)^\top$.
		Now, an admissible control $w$ in the neighborhood of $u$
		can visit $(0,1)^\top$ or $(1,1)^\top$
		without an increase of the switching cost $\TVp$.
		However, $(0,1)^\top$ and $(1,1)^\top$ cannot be represented as convex combinations of $(0,0)^\top$ and $(1,2)^\top$.
	\end{example}
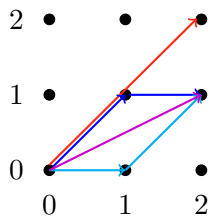
\begin{figure}[h]
	\centering	
	\begin{tikzpicture}
		% Points
		\foreach \x in {0,1,2}
		\foreach \y in {0,1,2}
		\filldraw (\x,\y) circle (2pt);
		% Numbers on x-axis
		\foreach \x in {0,1,2}
		\node[below] at (\x,-0.2) {\x};
		% Numbers on y-axis
		\foreach \y in {0,1,2}
		\node[left] at (-0.2,\y) {\y};
		\draw[->,BTUred,thick] (0,0.07) -- (1.94,2.01);

			\draw[->,blue,thick] (0.05,0.05) -- (1,1);
			\draw[->,blue,thick] (1.05,1) -- (2,1);
			\draw[->, cyan,thick] (0,0) -- (1,0);
			\draw[->, cyan,thick] (1,0) -- (2,1);
			\draw[->,violet,thick] (0.0,0.0) -- (2,1);
	\end{tikzpicture}
\caption{Switching from $(0,0)$ to $(2,2)$ (red vector) touches $(1,1)$, thus, stopping at $(1,1)$ does not lead to a higher variation. The blue, cyan and violet paths have the same length w.r.t. the 1- and $\infty$-norm.}
\label{fig:switching_paths}
\end{figure}

Thus,
we incorporate convex combinations of adjacent values $a_i$ and $a_{i+1}$
in the representation.
\begin{definition}
	\label{def:full_representation}
	Let $u \in \Uad$ be given.
	A triple $(\hat n, \hat a, \hat t)$
	with $\hat n \in \N$, $\hat a \in \VV^{\hat n}$, $\hat t \in \R^{\hat n - 1}$
	is called the full representation of $u$
	if it satisfies the following properties:
	\begin{enumerate}
		\item\label{def:full:1}
			$u = v^{\hat t, \hat a}$,
		\item\label{def:full:2}
			$\hat t_i \le \hat t_{i + 1}$ for all $i \in [n - 1]_0$ with $\hat t_0 = 0$, $\hat t_n = T$,
		\item\label{def:full:3}
			if $i, j \in [n]$ are given such that $i \le j$ and
			$\hat t_{i-1} < \hat t_i = \hat t_j < \hat t_{j + 1}$
			then
			we can write
			$\hat a_k = (1 - \lambda_k) \hat a_i + \lambda_k \hat a_{j+1}$,
			$k = i+1, \ldots, j$
			for a strictly increasing sequence $\lambda_{i+1}, \ldots, \lambda_{j} \in (0,1)$
			and
			this sequence contains every element from the set
			$\Lambda_{i,j} := \set{ \lambda \in (0,1) \mid (1 - \lambda) \hat a_i + \lambda \hat a_{j+1} \in \VV}$.
	\end{enumerate}
\end{definition}
Property \ref{def:full:3} serves multiple roles in \cref{def:full_representation}.
First, if there appear equal values in $\hat t$, then the corresponding values in $\hat a$ need to contain all convex combinations of $\hat a_i$ and $\hat a_{j+1}$ in the admissible set $\VV$ in correct order.
Second, if $i=j$ (i.e., $\hat t_i$ appears only once in $\hat t$), then no convex combination of $\hat a_i$ and $\hat a_{j+1} = \hat a_{i+1}$ belongs to $\VV$. Indeed, the mentioned sequence $\lambda_{i+1}, \ldots, \lambda_j$
is the empty sequence for $i = j$
and, consequently,
the set $\Lambda_{i,i}$ has to be empty in this case.
Finally, $i=j$ also implies that $\hat a_i \ne \hat a_{i+1}$, as otherwise we would have $\Lambda_{i,i} = (0,1)$.

Using the preceding arguments, it becomes clear that the full representation of a feasible $u$
is unique.
Note that for $M=1$, the representation is consistent with \cite[Notation 3.3]{MarkoWachsmuth2022:1}, which is why we will denote it in the same way.

Given the minimal representation of $u$,
one can construct the full representation
by duplicating time steps $t_i$
whenever there are convex combinations of $a_i$ and $a_{i + 1}$
belonging to $\VV$.
This is similar to the construction in the proof of \cite[Lemma~3.1]{MarkoWachsmuth2022:1}.

It is easy to check that using the data $(n,a,t)$ of the minimal representation $v^{t,a}$ of $u$, we can calculate $\TVp(u)$ with the pointwise $p$-variation by
\[\TVp(u) = \sum_{i=1}^{n-1} \norm*{a_{i+1}-a_i}_p.\]
We prove the statement for the full representation in a Lemma.
	\begin{lemma}
		\label{lem:TVp_full_repr}
		Let $u\in\Uad$ and $(\hat n,\hat a,\hat t)$ be the data of the full representation of $u$. Then, we have
		\[\TVp(u) = \sum_{k=1}^{\hat n-1}\norm*{\hat a_{k+1} - \hat a_k}_p.\]
		Thus $\TVp$ is independent of the switching time vector of the full representation.
	\end{lemma}
	\begin{proof}		
		By the definition of the full representation, if $\hat t_i = \hat t_j$ for $i<j$, then $\hat a_{i+1},\dots,\hat a_{j}$ are convex combinations of $\hat a_i$, $\hat a_{j+1}$. Thus, it can be quickly checked that
		\begin{equation}
			\label{proof:equ1}
			\norm{\hat a_i - \hat a_{j+1}}_p = \sum_{k=i}^j \norm{\hat a_k - \hat a_{k+1}}_p.
		\end{equation}
		Let $I$ be the index set containing $i+1,\dots,j$ for all $i<j$ with $\hat t_i = \hat t_j$ and let $\{i_1,\dots,i_n\}:= [\hat n-1]\setminus I$. Due to the fact that the statement is true for the minimal representation, we have
		\[\TVp(u) = \sum_{l=1}^{n-1} \norm{\hat a_{i_l} - \hat a_{i_{l+1}}}_p.\]
		Clearly, using \eqref{proof:equ1}, we can incorporate the indices of $I$ into this sum to get the statement.
	\end{proof}
	The following auxiliary result will be useful in the proof of the next theorem.
	\begin{lemma}
		\label{lem:similar_local_points}
		Let $u$ be a feasible point of \eqref{eq:prob}, $(\hat n,\hat t,\hat a)$ be the data of the full representation of $u$ and $p\in(1,\infty)$. 
		Then, there exist $\varepsilon,c,C > 0$ such that for all $w\in\Uad$
		with $\TVp(w) \le \TVp(u)$ and $\norm{w - u}_{L^1} \le \varepsilon$
		there exists $\hat s \in \R^{{\hat n} - 1}$
		with
		$0 <  \hat s_1 \le \ldots \le \hat s_{{\hat n}-1} < T$
		and
		$w = v^{\hat s,\hat a}$.
		Furthermore, we have $\TVp(w) = \TVp(u)$ and
		\begin{equation*}
			c \norm{\hat s - \hat t}_{2}
			\le
			\norm{w - u}_{L^1}
			\le
			C \norm{\hat s - \hat t}_{2}.
		\end{equation*}
	\end{lemma}
\begin{proof}
	Let $(n,a,t)$ be the data of the minimal representation of $u$. We define $\varepsilon$ the same way as $\hat \varepsilon$ in the proof of \cite[Lemma 3.4]{MarkoWachsmuth2022:1} and thus, get the existence of $(\alpha_j,\beta_j)\subset[t_j,t_{j+1}]$ for all $j\in[n-1]_0$ such that $w=u=a_{j+1}$ in $(\alpha_j,\beta_j)$ for all $w\in\Uad$ with $\norm{u-w}_{L^1}\leq \varepsilon$ and $\TVp(w)\leq\TVp(u)$. Let $\tilde t_j$ be the midpoint of $(\alpha_j,\beta_j)$. We define $\varphi\in C_c^1(\tilde t_j,\tilde t_{j+1})^M$ as a 
	continuous function such that $\norm{\varphi(\tau)}_{p'}\leq 1$ for all $\tau\in(\tilde t_{j+1},\beta_{j}]\cup [\alpha_{j+1},\tilde t_{j+1})$ and
	\[\varphi_i(\tau) = \frac{(a_{j+1,i}-a_{j+2,i})\abs{a_{j+1,i}-a_{j+2,i}}^{p-2}}{\norm{a_{j+1}-a_{j+2}}_p^{p/p'}}\quad\forall \tau\in(\beta_j,\alpha_{j+1}),\ i\in[M],\] where $\varphi_i$ and $a_{j,i}$ denote the $i$-th components of $\varphi$ and $a_j$ and $p'$ is the Hölder conjugate of $p$. Observe that for any $x\in\R$, we have $\abs*{x\cdot\abs{x}^{p-2}}^{p'} = \abs{x}^p$ and thus,  \[\norm*{\varphi(\tau)}_{p'}^{p'} = \frac{1}{\norm{a_{j+1}-a_{j+2}}_p^{p}}\sum_{i = 1}^M \abs*{a_{j+1,i}-a_{j+2,i}}^p= 1
	\quad\forall\tau\in (\beta_j,\alpha_{j+1}).\]
	 Then,
	\[\begin{aligned}
		\TVp(w,(\tilde t_j,\tilde t_{j+1}))&\geq \int_{\tilde t_j}^{\tilde t_{j+1}} w^\top\varphi'\,\d \tau = \int_{\tilde t_j}^{\beta_j} a_{j+1}^\top\varphi'\,\d \tau + \int_{\alpha_{j+1}}^{\tilde t_{j+1}} a_{j+2}^\top\varphi'\,\d \tau\\
		&= \frac{1}{\norm{a_{j+1}-a_{j+2}}_p^{p/p'}}\sum_{i = 1}^M \abs*{a_{j+1,i}-a_{j+2,i}}^p = \norm{a_{j+1}-a_{j+2}}_p^{p-p/p'}\\
		&=\norm{ a_{j+1}- a_{j+2}}_p=  \norm{ a_{j+2}- a_{j+1}}_p.
	\end{aligned}\]
	It follows that
	\[\begin{aligned}
\TVp(u)&\geq \TVp(w) = \abs{Dw}_p((0,T))\geq \sum_{j=0}^{n-2}\abs{Dw}_p((\tilde t_j,\tilde t_{j+1})) = \sum_{j=0}^{n-2}\TVp(w,(\tilde t_j,\tilde t_{j+1}))\\
&\geq \sum_{j=0}^{n-2}\norm{a_{j+2}-a_{j+1}}_p = \sum_{j=1}^{n-1}\norm{a_{j+1}-a_j}_p = \TVp(u).
	\end{aligned}\]
Thus, $\TVp(u) = \TVp(w)$ and $\TVp(w,(\tilde t_j,\tilde t_{j+1})) = \norm{a_{j+2}-a_{j+1}}_p$. 

Now, consider the minimal representation $(\bar n,\bar a,\bar t)$ of $w$. We know there is an $l\in[\bar n-1]$ and an $m>0$ such that $\bar a_l = a_j$, $\bar a_{l+m} = a_{j+1}$ for every $j\in[n-1]$. For any $j$ with $m>1$, we have
\[\TVp(w,(\tilde t_j,\tilde t_{j+1})) = \sum_{i = l}^{l+m-1}\norm{\bar a_{i+1}-\bar a_i}_p\overset{!}{=} \norm*{\sum_{i = l}^{l+m-1}(\bar a_{i+1}-\bar a_i)}_p=\norm{a_{j+2}-a_{j+1}}_p.\]
For $p\in(1,\infty)$, this can only be true if $\bar a_{l+1},\dots,\bar a_{l+m-1}$ are representable by a convex combination of $a_{j+1}$ and $a_j$. Thus, the full representation of $w$ is given by the tuple $(\hat n, \hat a,\hat s)$, i.e. $w=v^{\hat s,\hat a}$.  %\cref{lem:TVp_full_repr} now implies $\TVp(u) = \TVp(w)$.
	
The rest of the proof is analogous to the proof of \cite[Lemma 3.4]{MarkoWachsmuth2022:1},  in which $\abs{\hat a_{j+1}-\hat a_j}$  has to be replaced with $\norm{\hat a_{j+1}-\hat a_{j}}_1$,  where we utilize $\norm{\hat a_{j+1}-\hat a_{j}}_1\in \left[\max_{i,j\in[d]}\set{\norm{\nu_i-\nu_j}_1}\right]$ for $\hat a_{j+1}\neq \hat a_{j}$, $j\in[\hat n-1]$.
\end{proof}
As indicated by \cref{ex:Lambda}, this result does not hold for $p\in\{1,\infty\}$, as we cannot guarantee that $w$ adheres to the switching pattern $\hat a$.

To gain optimality conditions of \eqref{eq:prob},
we fix $n\in \N$ and $a=(a_1,\dots,a_n)$ with $a_i\in\R^M$ for all $i\in[n]$, 
and we consider the problem
\begin{equation}
	\label{eq:ST}
	\tag{ST($ n, a$)}
	\begin{aligned}
		\text{Minimize} \quad & F(v^{t, a}) \\
		\text{with respect to} \quad & t \in \R^{n-1},\\
		\text{such that} \quad
		& 0 \le t_1 \le \dots \le t_{n-1} \le T.
	\end{aligned}
\end{equation}
The feasible set of \eqref{eq:ST} will be denoted by \begin{equation*}
	\FF := \set{ t \in \R^{n - 1} \given 0 \le t_1 \le \dots \le t_{n - 1} \le T }.
\end{equation*}
Now, using the previous lemma, we can show that for $p\notin\{1,\infty\}$ local optimality of $\eqref{eq:prob}$ and $\eqref{eq:ST}$ are equivalent:
\begin{theorem}
	\label{thm:equivalence}
	Let $u \in \Uad$ be given and consider the data $(\hat n, \hat a, \hat t)$ of its full representation.
	Then, $u$ is locally optimal for \eqref{eq:prob} with $p\in(1,\infty)$ in $L^1(0,T)^M$
	if and only if $\hat t$ is locally optimal for \STfull.
	Moreover, $u$ satisfies a local quadratic growth condition for \eqref{eq:prob} with $p\in(1,\infty)$
	in $L^1(0,T)^M$
	if and only if a local quadratic growth condition is valid for \STfull\ at $\hat t$, i.e. the existence of constants $\varepsilon,\eta > 0$
	such that
	\begin{equation}
		\label{eq:quad_growth_P}
		F(w) + \beta \TVp(w)
		\ge
		F(u) + \beta \TVp(u) + \frac\eta2 \norm{w - u}_{L^1}^2
		\quad
		\forall w \in \Uad, \norm{w - u}_{L^1} \le \varepsilon
	\end{equation}
	is equivalent to the existence of constants $\tilde\varepsilon, \tilde\eta > 0$
	with
	\begin{equation}
		\label{eq:quad_growth_STfull}
		F(v^{\hat s,\hat a})
		\ge
		F(u) + \frac{\tilde\eta}2 \norm{\hat s - \hat t}_{2}^2
		\qquad
		\forall \hat s \in \FF,
		\norm{\hat s - \hat t}_{2} \le \tilde\varepsilon
		.
	\end{equation}
\end{theorem}
\begin{proof}
	The theorem can be proven analogously to \cite[Theorem 3.5]{MarkoWachsmuth2022:1} by some straightforward modification. In particular, \cref{lem:similar_local_points} is employed for the implication ``$\Leftarrow$''.
\end{proof}
Note that local optimality for $\eqref{eq:prob}$ also implies local optimality for \STfull\ in the cases $p\in\{1,\infty\}$. However, the other direction cannot be shown, since $\hat a$ in $v^{\hat t,\hat a}$ could be altered without changing $\TVp(u)$, see \cref{ex:Lambda}.

In order to get optimality conditions of first and second order, we need to assume that $F\colon L^1(0,T)^M\to \R$ is twice Fréchet differentiable, where we identify $F'(u)$ with $\nabla F(u)\in L^{\infty}(0,T)^M$ and $F''(u)$ with $\nabla^2 F(u)\in L^\infty((0,T)^2;\R^{M\times M})$ such that
\begin{equation*}
	F''(u)[v,w] = \int_0^T\int_0^T v(r)^\top \nabla^2 F(u)(r,s) w(s)\,\d r\,\d s\quad\forall v,w\in L^1(0,T)^M.
\end{equation*}

Optimality conditions for the one-dimensional version of \STfull\ (i.e., $M=1$) were already investigated in \cite{MarkoWachsmuth2022:1}. For the sake of brevity, we present the results for $M>1$ and highlight the differences in the derivation of these conditions in the proof of the following theorem.

\begin{theorem}
	\label{thm:opt_conditions}
	Let $u \in \BV(0,T)^M$ be feasible for \eqref{eq:prob}
	and we denote by $(n,a,t)$
	the minimal representation for $u$. Let $J$ be the set of all $j\in[n-1]$ where $\Lambda_j:=\set{\lambda\in(0,1)\given \lambda a_j + (1-\lambda)a_{j+1}\in\VV}$ is nonempty.
	We assume that $F \colon L^1(0,T) \to \R$
	is twice Fréchet differentiable
	with 
	$\nabla F(u) \in C^1([0,T])^M$
	and $\nabla^2 F(u) \in C([0,T]^2;\R^{M\times M})$ and
	we define
	$\mu_j := a_{j} - a_{j+1}$ for $j = 1,\ldots,n-1$.
	If $u$ is a local minimizer of \eqref{eq:prob} in $L^1(0,T)^M$, then the system
	\begin{subequations}
		\label{eq:SONC}
		\begin{align}
			\label{eq:SONC_1}
			\mu_j^\top\nabla F(u)(t_j) &= 0 \qquad\forall j = 1,\ldots, n-1, \\
			\label{eq:SONC_2}
			\mu_j^\top(\nabla F(u))'(t_j) &\ge 0 \qquad\forall j \in J, \\
			\label{eq:SONC_3}
			\sum_{j = 1}^{n-1} \mu_j^\top (\nabla F(u))'(t_j) \tau_j^2
			+
			\sum_{j,k = 1}^{n-1} \mu_j^\top \nabla^2 F(u)(t_j, t_k)\mu_k \tau_j \tau_k
			&\ge0
			\qquad\forall \tau \in \R^{n-1}
		\end{align}
	\end{subequations}
	is satisfied.
	
	Let $p\in(1,\infty)$. Then,
	$u$ is a local minimizer of \eqref{eq:prob} satisfying
	a quadratic growth condition in $L^1(0,T)^M$
	if and only if
	\begin{subequations}
		\label{eq:SOEC}
		\begin{align}
			\label{eq:SOEC_1}
			\mu_j^\top\nabla F(u)(t_j) &= 0 \qquad\forall j = 1,\ldots, n-1, \\
			\label{eq:SOEC_2}
			\mu_j^\top(\nabla F(u))'(t_j) &> 0 \qquad\forall j \in J, \\
			\label{eq:SOEC_3}
			\sum_{j = 1}^{n-1} \mu_j^\top (\nabla F(u))'(t_j) \tau_j^2
			+
			\sum_{j,k = 1}^{n-1} \mu_j^\top \nabla^2 F(u)(t_j, t_k)\mu_k \tau_j \tau_k
			&>0
			\qquad\forall \tau \in \R^{n-1} \setminus \set{0}
			.
		\end{align}
	\end{subequations}
\end{theorem}
\begin{proof}
	For some $\bar n \in \N$ and a feasible point $\bar t$ of $\STbar$ with $\bar a=(\bar a_1,\dots,\bar a_{\bar n})\in \VV^{\bar n}$, let $\TT_{\bar n, \bar t}:= \set[\big]{
		\tau \in \R^{\bar n-1}
		\given
		\forall k \in [\bar n-2]
		:
		\bar t_k = \bar t_{k+1}
		\;\Rightarrow\;
		\tau_k \le \tau_{k+1}
	}$. Imitating \cite[Theorem 3.6]{MarkoWachsmuth2022:1}, we get the Taylor expansion 
	\begin{align*}
		F(v^{\bar t + \tau,\bar a})
		&=
		F(v^{\bar t,\bar a})
		+
		\parens*{ \sum_{j = 1}^{\bar n-1} \bar \mu_j^\top \nabla F(v^{\bar t,\bar a})(\bar t_j) \tau_j }
		\\&\quad
		+
		\frac{1}{2} \parens*{
			\sum_{j = 1}^{\bar n-1} \bar \mu_j^\top (\nabla F(v^{\bar t,\bar a}))'(\bar t_j) \tau_j^2
			+
			\sum_{j,k = 1}^{\bar n-1} \bar \mu_j^\top \nabla^2 F(v^{\bar t,\bar a})(\bar t_j, \bar t_k)\bar \mu_k \tau_j \tau_k
		}
		\\&\quad
		+
		\oo( \norm{\tau}^2 )
		\qquad\text{as } \norm{\tau} \to 0,
	\end{align*}
where $\bar \mu_j = \bar a_j - \bar a_{j+1}$, $j\in[\bar n-1]$ and the conventions $\bar t_0 = \tau_0 = \tau_n = 0$, $\bar t_n = T$ are used. Assume $\bar t$ is a local minimizer of \hyperref[eq:ST]{\textup{(ST($\bar n, \bar a$))}}. Thus, applying the arguments from \cite[Lemma 3.7]{MarkoWachsmuth2022:1} gives
\begin{subequations}
	\label{eq:necessary}
	\begin{align}
		\label{eq:necessary_1}
		&\mathrlap{\forall j = 1,\ldots,\bar n-1:}
		&
		\bar \mu_j^\top\nabla F(v^{\bar t,\bar a})(\bar t_j)& = 0
		,
		\\
		\label{eq:necessary_2}
		&\forall \tau \in \TT_{\bar n,\bar t}:
		&
		\sum_{j = 1}^{\bar n-1} \bar \mu_j^\top (\nabla F(v^{\bar t,\bar a}))'(\bar t_j) \tau_j^2
		+
		\sum_{j,k = 1}^{\bar n-1} \bar \mu_j^\top \nabla^2 F(v^{\bar t,\bar a})(\bar t_j,\bar t_k)\bar \mu_k \tau_j \tau_k&\geq 0
		.
	\end{align}
\end{subequations}
Additionally,
if
\begin{subequations}
	\label{eq:sufficient}
	\begin{align}
		\label{eq:sufficient_1}
		&\mathrlap{\forall j = 1,\ldots, \bar n-1:}
		&
		\bar \mu_j^\top\nabla F(v^{\bar t,\bar a})(\bar t_j)& = 0
		,
		\\
		\label{eq:sufficient_2}
		&\forall \tau \in \TT_{\bar n,\bar t}\setminus\set{0}:
		&
		\sum_{j = 1}^{\bar n-1} \bar \mu_j^\top (\nabla F(v^{\bar t,\bar a}))'(\bar t_j) \tau_j^2
		+
		\sum_{j,k = 1}^{\bar n-1} \bar \mu_j^\top \nabla^2 F(v^{\bar t,\bar a})(\bar t_j, \bar t_k)\bar \mu_k \tau_j \tau_k&> 0,
	\end{align}
\end{subequations}
then $\bar t$ is a local minimizer of \STbar\ and a quadratic growth condition is satisfied.

Now, let $u\in \Uad$ and denote by $(\hat n,\hat a,\hat t)$ the full representation of $u$. Using \cref{thm:equivalence}, we can see that the conditions \eqref{eq:necessary} with $\hat n$, $\hat a$, $\hat t$ are necessary optimality conditions for local optimality w.r.t. \eqref{eq:prob} and in the case of $p\in(1,\infty)$ the conditions \eqref{eq:sufficient} are sufficient for local optimality and the validity of a local quadratic growth condition w.r.t. \eqref{eq:prob}.

In order to derive conditions involving the minimal representation, we introduce the matrices $\F\in \R^{(n-1) \times (n-1)}$ and $\hat \F\in \R^{(\hat n-1) \times (\hat n-1)}$ defined by 
\begin{align*}
	\tau^\top \F \tau
	&:=
	\sum_{j = 1}^{n-1} \mu_j^\top (\nabla F(u))'(t_j) \tau_j^2
	+
	\sum_{j,k = 1}^{n-1} \tau_j\mu_j^\top \nabla^2 F(u)(t_j, t_k) \mu_k \tau_k
	\qquad\forall \tau \in \R^{n-1},
	\\
	\hat\tau^\top \hat\F \hat\tau
	&:=
	\sum_{i = 1}^{\hat n-1} \hat\mu_i^\top (\nabla F(u))'(\hat t_i) \hat\tau_i^2
	+
	\sum_{i,l = 1}^{\hat n-1} \hat\tau_i\hat\mu_i^\top \nabla^2 F(u)(\hat t_i, \hat t_l) \hat\mu_l \hat\tau_l
	\qquad\forall \hat\tau \in \R^{\hat n-1},
\end{align*}
where $\hat\mu_i = \hat a_i - \hat a_{i+1}$, $i\in[\hat n-1]$. Also, set $I_j := \set{ i \in [\hat n-1] \given \hat t_i = t_j}$ therefore $\sum_{i \in I_j} \hat\mu_i = \mu_j$. By definition of the full representation, all $\hat a_i$ for $i\in I_j$ are convex combinations of $a_j$ and $a_{j+1}$ for all $j\in[n-1]$. Thus, there are $\lambda_i\in[0,1]$, $i\in I_j$ such that $\hat \mu_i = \lambda_i\mu_j$ and $\sum_{i\in I_j} \lambda_i = 1$ for all $j\in[n-1]$. Now, we can mimic the arguments of \cite[Lemma 3.9]{MarkoWachsmuth2022:1} by defining $\tau_j =
\sum_{i \in I_j} \lambda_i \hat\tau_i$, which implies that $\sum_{i \in I_j} \hat\mu_i \hat\tau_i = \mu_j\tau_j$ and therefore
\begin{align*}
	\hat\tau^\top \hat \F \hat\tau
	&=
	\sum_{j = 1}^{n-1} \parens[\Bigg]{ \sum_{i \in I_j} \hat\mu_i \hat\tau_i^2 }^\top(\nabla F(u))'(t_j)
	+
	\sum_{j,k = 1}^{n-1}\parens[\Bigg]{ \sum_{i \in I_j} \hat\mu_i \hat\tau_i }^\top \nabla^2 F(u)(t_j, t_k)
	\parens[\Bigg]{ \sum_{l \in I_k} \hat\mu_l \hat\tau_l }
	\\&
	=
	\sum_{j = 1}^{n-1} 
	\parens[\Bigg]{ \sum_{i \in I_j} \hat\mu_i \hat\tau_i^2 - \mu_j \tau_j^2}^\top
	(\nabla F(u))'(t_j)
	+
	\tau^\top \F \tau \\
	&= \sum_{j = 1}^{n-1} 
	\parens[\Bigg]{ \sum_{i \in I_j} \lambda_i \hat\tau_i^2 - \tau_j^2}\mu_j^\top
	(\nabla F(u))'(t_j)
	+
	\tau^\top \F \tau.
\end{align*}
Using the Jensen inequality, we can see that $\sum_{i \in I_j} \lambda_i \hat\tau_i^2  \geq \tau_j^2$ and the arguments in the proof of \cite[Lemma 3.9]{MarkoWachsmuth2022:1} lead to the equivalences
\begin{align*}
	\hat\tau^\top \hat\F \hat\tau \ge 0
	\quad\forall \hat\tau \in \hat\TT_{\hat n,\hat t}
	\qquad\Leftrightarrow\qquad
	\F \succeq 0
	&\;\land\;
	\mu_j^\top(\nabla F(u))'(t_j) \ge 0 \quad\forall j \in J
\end{align*}
and
\begin{align*}
	\hat\tau^\top \hat\F \hat\tau > 0
	\quad\forall \hat\tau \in \hat\TT_{\hat n,\hat t}\setminus\set{0}
	\qquad\Leftrightarrow\qquad
	\F \succ 0
	&\;\land\;
	\mu_j^\top(\nabla F(u))'(t_j) > 0 \quad\forall j \in J
	.
\end{align*}
Finally, combining these results with \eqref{eq:necessary} and \eqref{eq:sufficient} gives us \eqref{eq:SONC} and \eqref{eq:SOEC}.
\end{proof}
\begin{remark}
	The conditions \eqref{eq:SONC} and \eqref{eq:SOEC} are a generalization of (3.10), (3.11) from \cite{MarkoWachsmuth2022:1}, since $J = J^+\cup J^-$ and $j\in J^\pm$ implies $\sgn(\mu_j) = \mp1$. To avoid confusion, note that there is an error in the sign of $\mu_j$ in \cite[Theorem 3.6]{MarkoWachsmuth2022:1} and consequently, in all succeeding results of Section 3.2. Indeed, the first equation in the proof of Theorem 3.6 is only correct with $\mu_j:=a_j-a_{j+1}$.
	
	More generally, it is enough to assume that $\nabla F$ is continuosly differentiable and that $\nabla^2 F$ is continuos only in neighborhoods of the switching points $t_i$, $i\in[n-1]$, as these conditions are required solely for the Taylor expansion.
\end{remark}
	\section{Trust-region Method and Efficient Solution of Subproblems}
	\label{sec:method}
	We adapt the trust-region algorithm proposed in \cite[Chapter 3.1]{LeyfferManns2021}, where the 
	objective is partially linearized around a given feasible point. When employing such an algorithm, one has to solve subproblems of the form
	\begin{equation}
		\label{eq:TR_subprob}
		\tag{TR($v,g,\Delta^k$)}
		\begin{aligned}
			\text{Minimize} \quad & (g,u - v)_{L^2(0,T)^M} + \beta \TVp(u) - \beta\TVp(v) \\
			\text{such that} \quad & \norm{u-v}_{L^1}\leq \Delta^k,\quad u \in  \Uad
		\end{aligned}
	\end{equation}
	with a given function $v\in\Uad$ and $g=\nabla F(v)\in L^2(0,T)^M$. 
	
	In the case of $M=1$, it was shown in \cite[Proposition 3.2]{LeyfferManns2021} that \eqref{eq:TR_subprob} admits a minimizer. It is not hard to see that this result also applies for $M>1$.
	Now, \cite[Algorithm~1]{LeyfferManns2021}
	can be used verbatim and in Step~6 of this algorithm, the above subproblem has to be solved.
	However, the analysis of the algorithm requires more attention, as the results of \cite[Sections 4.3, 4.4]{LeyfferManns2021} cannot be applied without further comments.

	\subsection{Convergence Analysis}
	\label{sec:convergence}
	In this section, we want to verify that the results from \cite[Theorem 4.23]{LeyfferManns2021} still hold in our setting. When checking the proofs leading up to the theorem, it can be verified that it is essentially only necessary to adapt Lemma 4.8 to the situation that \eqref{eq:SONC_1} is violated and to notice that if $\TVp(w_1) - \TVp(w_2)\neq 0$ for $w_1,w_2\in\Uad$, we have $\abs*{\TVp(w_1) - \TVp(w_2)}\geq 1$, as there is no path with a length smaller than 1 between two different points in the grid $\Z^M$ w.r.t. any $p$-norm.
	
	We assume that $F$ is twice Fréchet-differentiable, $\nabla F(v)\in C^1(0,T)^M$ for all $v\in \Uad$ and that $F''(\xi)$ is uniformly bounded on $\left(L^1(0,T)^M\right)^2$ for all $\xi\in L^2(0,T)^M$, i.e. there exists $C\geq 0$ such that
	\[\abs*{F''(\xi)(u,w)}\leq C\norm{u}_{L^1}\norm{v}_{L^1}\quad\forall u,v\in L^1(0,T)^M.\]
	
	\begin{lemma}
		\label{lem:decrease_non_stationary}
		Let $u=v^{t,a}\in\Uad$ be given, where $v^{t,a}$ is the minimal representation of $u$. Assume that $u$ does not satisfy \eqref{eq:SONC_1}, i.e. there exists an $i\in[n-1]$ with $\mu_i^\top\nabla F(u)(t_i)\neq 0$. Then, there exist $\varepsilon>0$ and $h_0>0$ such that for all $h\leq h_0$, there is $d^h\in L^1(0,T)^M$ with $\norm{d^h}_{L^1} =h$, $u+d^h\in\Uad$ and $\innerprod{\nabla F(v)}{d^h}_{L^2(0,T)^M}\leq -\frac{\varepsilon}{2}h$.
	\end{lemma}
	\begin{proof}
		First consider the case that $\mu_i^\top\nabla F(u)(t_i)> \varepsilon$ for some $\varepsilon>0$. Then, there is an $h_0>0$ such that for all $h\leq h_0$, we have
		\[\int_{t_i-h}^{t_i} \mu_i^\top\nabla F(u)(s)\,\d s>\frac{\varepsilon}{2}h.\]
		For $\tilde h:= h/\norm{\mu_i}_1$, choosing $d^h:= -\chi_{[t_{i}-\tilde h,t_i)}\mu_i$ gives us the statement.
		
		Likewise, if $\mu_i^\top\nabla F(u)(t_i)<- \varepsilon$ for some $\varepsilon>0$, there is an $h_0>0$ such that for all $h\leq h_0$
		\[\int_{t_i}^{t_i+h} \mu_i^\top\nabla F(u)(s)\,\d s<-\frac{\varepsilon}{2}h.\]
		For $\tilde h:= h/\norm{\mu_i}_1$, we can similarly choose $d^h:= \chi_{[t_{i},t_i+\tilde h)}\mu_i$.
	\end{proof}
	
	\begin{theorem}
		 \label{thm:convergence_TR}
		Let the iterates $(u^n)_{n\in\N}$ be produced by the trust region algorithm from \cite[Algorithm 1]{LeyfferManns2021}, where $\TV$ is replaced by $\TVp$. Then, $u^n\in \Uad$ for all $n\in\N$ and the sequence of objective values $(F(u_n) + \beta\TVp(u_n))_n$ is monotonously decreasing. Furthermore, exactly one of the following statements is true:
		\begin{enumerate}
			\item The sequence $(u^n)_n$ is finite and the final element $u^N$ solves \hyperref[eq:TR_subprob]{\textup{(TR($u^N,\nabla F(u^N), \Delta$))}} for some $\Delta>0$ and satisfies \eqref{eq:SONC_1}.
			\item The sequence $(u^n)_n$ is finite and the inner loop does not terminate for the final element $u^N$, which satisfies  \eqref{eq:SONC_1}.
			\item The sequence $(u^n)_n$ has a weak-$\star$ accumulation point in $\BV(0,T)$. Every weak-$\star$ accumulation point $v$ of $(u^n)_n$ is feasible, satisfies \eqref{eq:SONC_1} and for every subsequence $(u^{n_l})_{l\in \N}$ with $u^{n_l}\weaklystar v$ in $\BV(0,T)$ we have $u^{n_l}\to v$ in $L^1(0,T)^M$ and $\TVp(u^{n_l})\to \TVp(v)$ for all $p\in[1,\infty]$. Furthermore, if the trust region radii are bounded away from zero for a subsequence $(u^{n_k})_{k\in\N}$ and $\bar v$ is a weak-$\star$ accumulation point of $(u^{n_k})_k$, then $\bar v$ solves \hyperref[eq:TR_subprob]{\textup{(TR($\bar v, \nabla F(\bar v), \Delta$))}} for some $\Delta>0$.
		\end{enumerate}
	\end{theorem}
	\begin{proof}
		We follow the proofs of \cite[Lemma 4.11, Lemma 4.19, Theorem 4.23]{LeyfferManns2021} with the notation $v^n\gets u^n$, $\alpha\gets \beta$, $\mathcal{F}_{(P)} \gets\Uad$. Also, whenever L-stationarity is mentioned, we will consider \eqref{eq:SONC_1} and use \cref{lem:decrease_non_stationary} instead of \cite[Lemma 4.8]{LeyfferManns2021}. 
		
		Note that we will get
		\[\abs{\innerprod{\nabla F(u^{n_l})}{u^{n_l}-\tilde v}_{L^2(0,T)^M}}\leq \sqrt{T\max_{i,j}\norm{\nu_i-\nu_j}_2}\sqrt{\Delta^{n_{l+1},k}}\norm{\nabla F(v^{n_l})}_{L^2}\]
		instead of the inequality below \cite[(4.5)]{LeyfferManns2021}, which however does not alter the argument. Finally, we mention that \cite[Proposition A.2]{LeyfferManns2021} holds with $M_1 = \sqrt{\max_i \norm{\nu_i}_1}$ and $M_2 = \sqrt{\max_{i,j}\norm{\nu_i-\nu_j}_1}$.
	\end{proof}
	We cannot expect to get any information regarding the validity of \eqref{eq:SONC_2},  \eqref{eq:SONC_3}, as the trust region method only uses first order information while generating a solution of \eqref{eq:prob}.
	
	\subsection{Employing Bellman's Optimality Principle for \eqref{eq:TR_subprob}}
	\label{sec:bellman}
	We will discretize the problem \eqref{eq:TR_subprob} and solve the discretization numerically in a very efficient way using Bellman's optimality principle.
	First, we rephrase the problem by omitting terms in the objective independent of $u$, obtaining the formulation
	\begin{equation}
		\label{eq:TR_subprob2}
		\tag{TR2}
		\begin{aligned}
			\text{Minimize} \quad & (g,u)_{L^2(0,T)^M} + \beta \TVp(u)\\
			\text{such that} \quad & \norm{u-v}_{L^1(0,T)}\leq \Delta^k,\quad u \in  \Uad.
		\end{aligned}
	\end{equation}
	
	Next, we introduce a grid $0=t_0<\dots<t_n = T$ on $[0,T]$ for some $n\in\N$ such that $t_i-t_{i-1} = \frac Tn =: \Delta t$ for all $i\in[n]$. We interpret $u$, $v$ and $g$ as piecewise constant functions, i.e.\ $u=\sum_{i=1}^n u^i\chi_{[t_{i-1},t_i)}$ with $u^i\in \VV$, the representations of $g$ and $v$ are analogous with $v^i$ and $g^i$ being the (componentwise) mean values of $v$ and $g$ over $[t_{i-1},t_i]$. We can represent $u^i$ by $\nu_{l^i}$ for some $l^i\in[d]$. Thus, the discrete solution can be determined by finding the best decisions $l^i$ in every timestep.
	
	Similarly to \cite{MarkoWachsmuth2022:1}, we introduce the 
	so-called budget $B:=\lfloor \Delta^k/\Delta t \rfloor\in\N$, which marks an upper bound for $\sum_{j=1}^{n}\norm{u^j-v^j}_{1}$. Now, the discretized formulation of \eqref{eq:TR_subprob2} reads 
	\begin{equation}
		\label{eq:TR_discr}
		\tag{TR3}
		\begin{aligned}
			\text{Minimize} \quad & \Delta t\sum_{j=1}^{n} {g^j}^\top \nu_{l^j}  + \beta \sum_{j=1}^{n-1}\norm{\nu_{l^{j+1}}-\nu_{l^j}}_p\\
			\text{such that} \quad & \sum_{j=1}^{n} \norm{\nu_{l^j}-v^j}_1\leq B,\quad l^j\in [d]\;\forall j\in[n].
		\end{aligned}
	\end{equation}
	We define the value function 
	\begin{equation*} 
		\label{eq:TR_value_func1}
		\Phi_{l, i, b}
		:=
		\min
		\set*{
			\Delta t\sum_{j=i}^{n} {g^j}^\top \nu_{l^j}  + \beta \sum_{j=i}^{n-1}\norm{\nu_{l^{j+1}}-\nu_{l^j}}_p
			\given
			\begin{aligned}
				& l^i = l,\\
				&l^{i+1},\dots,l^{n}\in [d],\\
				& \sum_{j=i}^{n}\norm{\nu_{l^j}-v^j}_1 = b
			\end{aligned}
		}
	\end{equation*}
	for $b\in[B]_0$, $l\in [d]$ and $i\in[n]$ with the convention $\min\emptyset :=\infty$. The sums in $\Phi$ can be split, which allows us to rewrite the function as
	\begin{equation*} 
		\Phi_{l, i, b}
		=
		\min
		\set*{
			\Delta t {g^i}^\top \nu_l  + \beta \norm{\nu_{l^{i+1}}-\nu_l}_p + \Phi_{l^{i+1},i+1,b-\tilde b}
			\given
			\begin{aligned}
				& l^{i+1} \in [d],\\
				& \norm{\nu_l-v^i}_1 = \tilde b\leq b
			\end{aligned}
		}
	\end{equation*}
	for $i\in[n-1]$. Now, the solution of \eqref{eq:TR_discr} can be gained by calculating $\Phi_{l,i,b}$ for $i=n-1,\dots,1$ and all values of $l\in [d]$, $b\in[B]_0$, starting with
	\[\Phi_{l,n,b} := \begin{cases}
		\Delta t g^n\nu_l\,&\text{if } b = \norm{\nu_l-v^n}_1,\\
		+\infty\,&\text{else}.
	\end{cases}\]
	The corresponding minimizer in the evaluation of $\Phi_{l,i,b}$ is saved in an array $U\in\R^{d\times (n-1)\times B}$, i.e.,
	\begin{equation*} 
		U_{l, i, b}
		:=
		\argmin
		\set*{
			\Delta t {g^i}^\top \nu_l  + \beta \norm{\nu_{l^{i+1}}-\nu_l}_p + \Phi_{l^{i+1},i+1,b-\tilde b}
			\given
			\begin{aligned}
				& l^{i+1} \in [d],\\
				& \norm{\nu_l-v^i}_1 = \tilde b\leq b
			\end{aligned}
		},
	\end{equation*}
	with the convention $\argmin\emptyset := 0$.
	This array $U$ can be used to reconstruct the solution after $\Phi_{l,1,b}$ is calculated: Starting with 
	\begin{equation}
		\label{eq:argmin_phi}
		(\hat l^1,b_1) = \argmin_{l\in [d],b\in[B]_0}\Phi_{l,1,b},
	\end{equation}
	we set 
	\begin{equation}
	\label{eq:U_recursive}
	\hat l^{i+1} = U_{\hat l^i,i,b_i},\quad b_{i+1} = b_i - \norm{\nu_{\hat l^i}-v^i}_1	
	\end{equation}
	for all $i\in[n]$.
	The solution of \eqref{eq:TR_discr} is now given by $(\hat l^1,\dots,\hat l^n)$.
	\begin{remark}
		\label{rem:trust-subproblem}
		The above approach is a generalization of the subproblem solver from \cite[Chapter 5]{MarkoWachsmuth2022:1}. In the same way, the procedure has a runtime of $\OO(d^2n^2)$.
		Furthermore, only a $d\times 2\times B$ array $\Phi$ is needed when the above calculations are carried out and the generalization to non-equidistant meshes is only possible when the occurring interval lengths $t_j - t_{j-1}$
		are integer multiples of a minimal length. 
	\end{remark}
	The procedure may be generalized for problems with a constraint bounding the number of switches. Indeed, if the number of switches is bounded by $\sigma_{\text{max}}>0$, we need to calculate $\Phi_{l, i, b, \sigma}$ for every $b\in[B]_0$, $l\in [d]$, $i\in[n]$ and $\sigma\in[\sigma_{\text{max}}]$, where $\sigma$ is the number of switches up until $i$ (i.e., the number all $j\in \{i+1,\dots,n-1\}$ with $l^j\neq l^{j+1}$). Thus, $U$ and $\Phi$ will gain an additional dimension and the runtime will increase to $\OO(d^2n^2\sigma_{\text{max}})$. This works similarly for componentwise switching bounds.
	\section{Numerical Examples}
	\label{sec:examples}
		We will investigate the performance of the trust-region algorithm exemplary on the Lotka Volterra multimode fishing problem governed by an ODE and on a heat-distribution problem governed by a PDE. The performance of the trust region algorithm was tested with 500 samples for each problem and different grid sizes $n$ using the parameters $\Delta^0 = 2$, $\sigma = \frac12$ and randomly generated initial guesses constructed such that there are no more than $n/10$ switches across all component functions. In the following, we will present the distribution of objective values and calculation times. Furthermore, for the output $u$ of the trust region algorithm, we will verify that $\norm{(\mu_i^\top\nabla F(u)(t_i))_{i\in[n_t-1]}}_2$ is indeed close to zero (as suggested by \cref{thm:convergence_TR}), where $(n_t,a,t)$ is the minimal representation of $u$.
		
		The algorithm is written in Julia Version 1.10.0 and all results are calculated with an Intel(R) Core(TM) i9-10900 CPU @ 2.80GHz on a Linux OS. In the first example, all occurring ODEs where solved using an explicit Euler scheme. In the second example, the problem was discretized by employing a finite element approach with Lagrange elements of degree 2 and 545 degrees of freedom. The corresponding time-dependent ODEs where solved with an implicit Euler method.
		
	\subsection{Lotka Volterra Multimode Fishing Problem}
	For parameter vectors $v_1 = (0.2,0.4,0.01)^\top$, $v_2= (0.1,0.2,0.1)^\top$ and $\beta = 10^{-3}$, $T=12$, we consider
	\begin{equation}
		\label{eq:LVM}
		\tag{LVM}
		\begin{aligned}
			\text{Minimize} \quad & \frac 12\int_0^T (y_1(t)-1)^2 + (y_2(t)-1)^2 + \beta\TV_\infty(u)\\
			\text{such that} \quad & y_1'(t) =y_1(t) - y_1(t)y_2(t) -y_1(t)\cdot v_1^\top u(t),\\
			&y_2'(t) = -y_2(t) + y_1(t)y_2(t) - y_2(t)\cdot v_2^\top u(t)\\
			&y(0) = (0.5,0.7)^\top,\\
			&\sum_{i=1}^3 u_i(t) = 1,\\
			&u(t)\in\set{0,1}^3\,\text{ a.e. on } (0,T).
		\end{aligned}
	\end{equation}
	The different controls model different fishing equipment, where catching from the predator or the prey population can be prioritized. This problem is used for numerical testing in \cite{BestehornHansknechtKirchesManns2020}, \cite{ZeileWeberSager2022} and is the vector-valued variant of (LV) from \cite[Section 6.1]{MarkoWachsmuth2022:1}. Similarly to \cite{MarkoWachsmuth2022:1}, we argue that \eqref{eq:LVM} satisfies our regularity assumptions. Note that it is no generalization due to the constraint $\sum_{i = 1}^3 u_i(t) = 1$, which means that it is not admissible not to fish. Without total variation regularization, the problem exhibits a chattering behaviour, meaning the optimal solution only exists in the limit of infinite switching.
	
	The numeric results can be observed in \cref{tab:LVM} and the violin plots showcasing the distributions of objective values, calculation times and the values of $\norm{(\mu_i^\top\nabla F(u)(t_i))_i}_2$ are displayed in \cref{fig:Violin_LVM}.
	
\begin{table}[htp]
	\centering
	\begin{tabular}{r|c|c|c|c}
		$n$  & range of objectives & range of times [$s$] &  average objective & average time [$s$] \\ \hline
		512 & $[0.9504,1.05]$     & $[0.1208,1.6795]$    & $0.9683$  & $0.5923$ \\
		1024 & $[0.9352,0.9927]$     & $[0.6912,4.3112]$ & $0.9412$ & $1.8987$\\
		2048  & $[0.9279,0.9746]$     & $[2.1305,11.881]$ & $0.931$ & $5.3815$ \\
		4096  & $[0.9243,0.9369]$     & $[6.8965,50.1027]$ & $0.9263$ & $16.643$ 
	\end{tabular}
	\caption{Results of applying the trust region method $500$ times to \eqref{eq:LVM} with random starting point $u_0$ for 
		different grid sizes $n$}
	\label{tab:LVM}
\end{table}

\begin{figure}[h]
	\includegraphics[width = .485\textwidth]{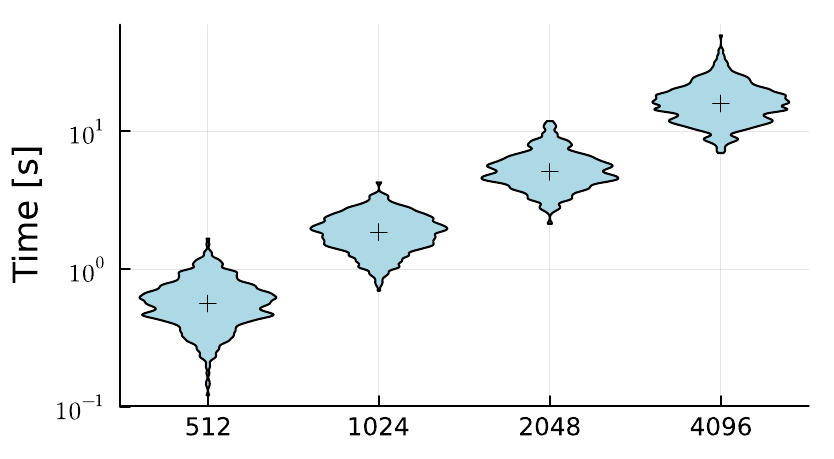}
	\includegraphics[width = .485\textwidth]{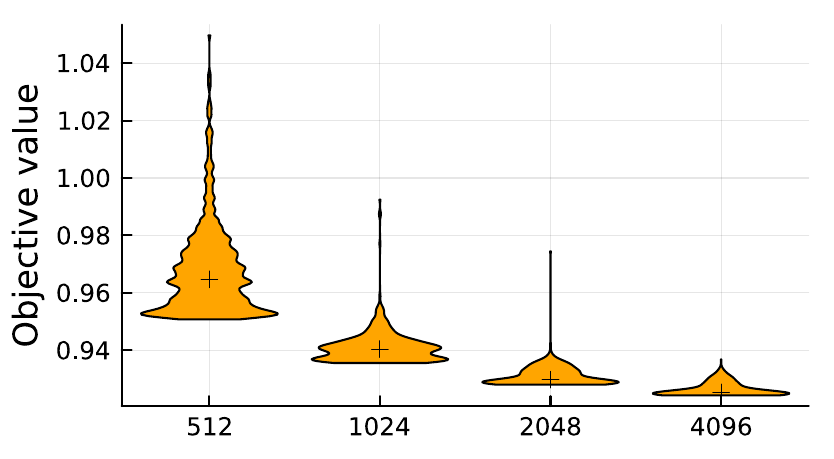}
	\vspace{.1cm}
	\begin{center}
		\begin{tikzpicture}
			\rotatebox{90}{
				\node at (0,0) {\footnotesize$\norm{(\mu_i^\top\nabla F(u)(t_i))_{i\in[n_t-1]}}_2$};}
			\node at (3.75,0) {\includegraphics[width = .485\textwidth]{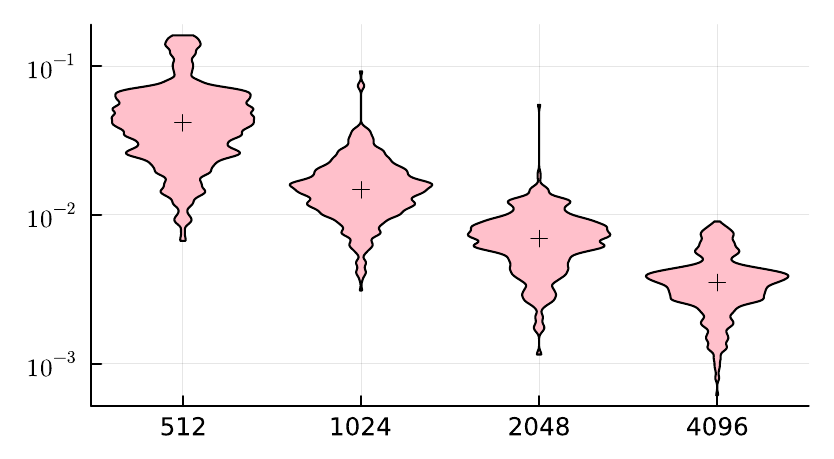}};
		\end{tikzpicture}
	\end{center}
	
	\caption{Distributions of the computation times, objective values and $\norm{(\mu_i^\top\nabla F(u)(t_i))_{i\in[n_t-1]}}_2$ for the trust region algorithm applied to \eqref{eq:LVM} for different grid sizes $2^k$, $k\in\{9,\dots,12\}$ and 500 samples. Cross represents the median.}
	\label{fig:Violin_LVM}
\end{figure}

We can see that for finer grids, the objective value will generally be smaller, and \eqref{eq:SONC_1} is more likely to be satisfied, although at the cost of a higher computation time. Indeed, by observing $\nabla F$ at the switches of $u$ in \cref{fig:sol_LVM}, we can visually verify \eqref{eq:SONC_1}. 

The computation time consists mainly in solving \eqref{eq:TR_subprob2} and in calculating the objective and derivative values. In every iteration of the outer loop, the trust region subproblem needs to be solved only once, as halving the trust region radius $\Delta^k$ (in the case of $\pred> 0$, $\ared < \sigma \pred$) leads to a halved budget. Consequently, the solution can be found using \eqref{eq:argmin_phi} and \eqref{eq:U_recursive} by replacing $B$ with $B/2$.

The derivative at the current iterate also is computed only once every outer iteration. However, the objective value needs to be evaluated in every inner iteration to calculate the actual reduction.
In the case of the above tested samples, the proportion of objective and derivative calculation time is showcased in \cref{tab:times_LVM}. One can see that about 93\% of the computing time is spent on solving the trust region subproblem \eqref{eq:TR_subprob2}.
	
\begin{table}[htp]
	\centering
	\begin{tabular}{r|c|c|c|c}
		$n$  & 512 & 1024 &  2048 & 4096 \\ \hline
		range [\%] & $[3.8,57.1]$     & $[4.7,34.5]$    & $[5.3,20.9]$  & $[5.5,14.4]$ \\
		average [\%] & $7.5$     & $7.7$ & $7.6$ & $6.8$
	\end{tabular}
	\caption{Percentage of computation time for the objective and derivative w.r.t. the total computation time for solving \eqref{eq:LVM} - range and average.}
	\label{tab:times_LVM}
\end{table}

We mention that the optimal objective value of the relaxed problem (with $u(t)\in[0,1]^3$) without the term $\beta\TVp(u)$ is roughly 0.91437 (cf. \href{https://mintoc.de/index.php/Lotka_Volterra_Multimode_fishing_problem}{mintoc.de}, \cite{Sager2011}). 

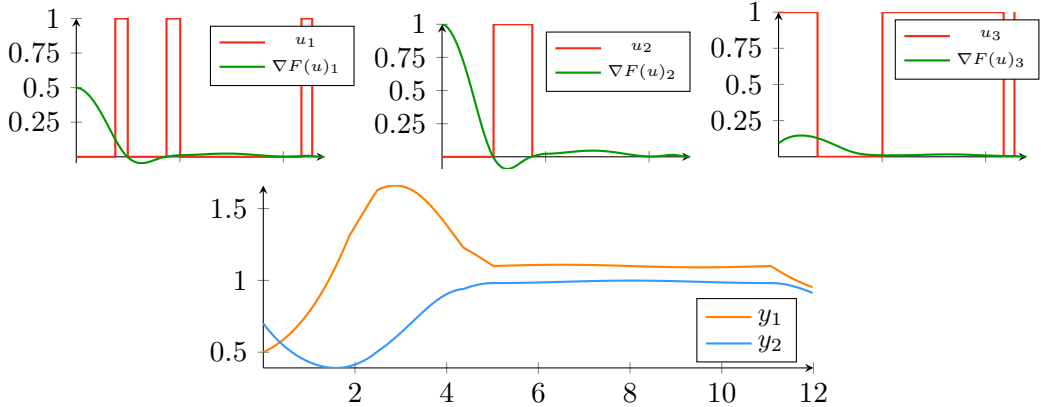
\begin{figure}[h]
	\centering
	\begin{tikzpicture}
		\begin{axis}[
			xmin=0, xmax=12,
			height = 3.5cm,
			width = .33\textwidth,
			axis lines=middle,
			ytick = {-.25,0,.25,.5,.75,1},
			xticklabel = \empty,
			legend style={at={(1.12,.95)}}
			]
			\addplot+[const plot, mark=none, BTUred,thick] file[skip first] {data_files/u_1__LVM_switch.dat};
			\addplot+[BTUgreen, mark = none,thick] file[skip first] {data_files/nabla_f_1__LVM_new.dat};
			\legend{\tiny $u_1$,\tiny $\nabla F(u)_1$}
		\end{axis}
	\end{tikzpicture}
	\begin{tikzpicture}
		\begin{axis}[
			xmin=0, xmax=12,
			width = .33\textwidth,
			height = 3.5cm,
			axis lines=middle,
			ytick = {-.25,0,.25,.5,.75,1},
			xticklabel = \empty,
			legend style={at={(1.,.95)}}
			]
			\addplot+[const plot, mark=none, BTUred,thick] file[skip first] {data_files/u_2__LVM_switch.dat};
			\addplot+[BTUgreen, mark = none,thick] file[skip first] {data_files/nabla_f_2__LVM_new.dat};
			\legend{\tiny $u_2$,\tiny $\nabla F(u)_2$}
		\end{axis}
	\end{tikzpicture}
	\begin{tikzpicture}
		\begin{axis}[
			xmin=0, xmax=12,
			width = .33\textwidth,
			height = 3.5cm,
			axis lines=middle,
			ytick = {-.25,0,.25,.5,.75,1},
			xticklabel = \empty,
			legend style={at={(1.05,.95)}}
			]
			\addplot+[const plot, mark=none, BTUred,thick] file[skip first] {data_files/u_3__LVM_switch.dat};
			\addplot+[BTUgreen, mark = none,thick] file[skip first] {data_files/nabla_f_3__LVM_new.dat};
			\legend{\tiny $u_3$,\tiny $\nabla F(u)_3$}
		\end{axis}
	\end{tikzpicture}
	\begin{tikzpicture}[>=latex]
		\begin{axis}[
			xmin=0, xmax=12,axis lines=middle,
			height = 4cm,
			width = .6\textwidth,
			legend pos=south east,
			]
			\addplot+[BTUorange,mark=none,thick] file[skip first] {data_files/y_1__LVM_new.dat};
			\addplot+[BTUlightblue, mark=none,thick] file[skip first] {data_files/y_2__LVM_new.dat};
			\legend{\small $y_1$, \small $y_2$}
		\end{axis}
	\end{tikzpicture}
	\caption{Solution $u$ of \eqref{eq:LVM} with scaled derivatives such that $\norm{\nabla F(u)}_{L^\infty} = 1$ and corresponding states $y_1$, $y_2$. Objective value: 0.9243. }
		\label{fig:sol_LVM}
\end{figure}

\subsection{Heat Distribution Problem}

We investigate a problem governed by the heat equation over the spatial domain $\Omega = (-1,1)^2$
and terminal time $T=10$. 
\begin{equation}
	\label{eq:Heat}
	\tag{HEAT}
	\begin{aligned}
		\text{Minimize} \quad & \int_0^T\norm{y(t,\cdot)-y_d(t,\cdot)}_{L^2(\Omega)}^2 + \gamma\sum_{i=1}^2 u_i(t)\,\d t + \beta\TV_2(u)\\
		\text{such that} \quad & \frac{\partial}{\partial t}y(t,x) - \alpha\Delta y(t,x) = \sum_{i = 1}^2 f_i(x)u_i(t)\quad \text{for a.a. }(x,t)\in (0,T)\times\Omega,\\
		&\nabla_x y \cdot \mathrm n = \kappa(y_{\mathrm{out}}-y) \quad\text{a.e. on }(0,T)\times \Gamma,\\
		&y(0, \cdot) = y_0 \quad\text{ a.e. on }\Omega,\\
		&u(t)\in\set{0,1,2,3,4,5}^2\,\text{ for a.a. } t \in (0,T).
	\end{aligned}
\end{equation}
Here, $y_d\in L^2((0,T)\times\Omega)$, $y_0\in L^2(\Gamma)$ and $y_{\mathrm{out}}\in L^2((0,T)\times\Gamma)$ denote the target, initial and environment temperature distributions. We set $\alpha = 0.5$, $\beta = 10^{-1}$, $\gamma= 10$, $\kappa = 0.12$, $y_d \equiv 20$, $y_0\equiv 10$ and $y_{\mathrm{out}} \equiv 0$. A Robin boundary condition on $\Gamma = \partial\Omega$ models the diffusion to the outside with insulation properties represented by $\kappa$. The functions $f_1, f_2\in L^2(\Omega)$ given by
\[f_i(x) = a_i e^{-b_i\norm{x-x_i}_2^2},\quad i=1,2,\quad a_i,b_i>0\]
with $a_1=a_2 = 20$, $b_1 = b_2 = 10$ are chosen to simulate heaters at $x_1 = (-1,0)$ and $x_2 = (1,0)$ controlled by the heating levels $u_1,u_2\in\{0,\dots,5\}$. Finally, the objective requires heating the domain from $y_0$ to the target temperature distribution $y_d$ while expending as little energy through the heaters as possible.

The validity of the regularity assumptions for problems of the form \eqref{eq:prob} is shown in \cref{sec:Appendix2}. Numeric results are displayed as in the previous section in \cref{tab:HEAT} and \cref{fig:Violin_HEAT}.

\begin{table}[htp]
	\centering
	\begin{tabular}{r|c|c|c|c}
		$n$  & range of objectives & range of times [$s$] &  average objective & average time [$s$] \\ \hline
		128 & $[807.2,816.4]$     & $[1.341,4.581]$    & $807.6$  & $3.738$ \\
		256 & $[808.0,814.1]$     & $[4.919,16.569]$    & $807.2$ & $12.409$\\
		512  & $[808.2,811.9]$     & $[14.485,76.184]$ & $806.9$ & $49.252$ \\
		1024  & $[808.5,810.0]$     & $[57.375,149.629]$ & $808.9$ & $93.428$ 
	\end{tabular}
	\caption{Results of applying the trust region method $500$ times to \eqref{eq:Heat} with random starting point $u_0$ for 
		different temporal grid sizes $n$}
	\label{tab:HEAT}
\end{table}

\begin{figure}[htp]
	\includegraphics[width = .485\textwidth]{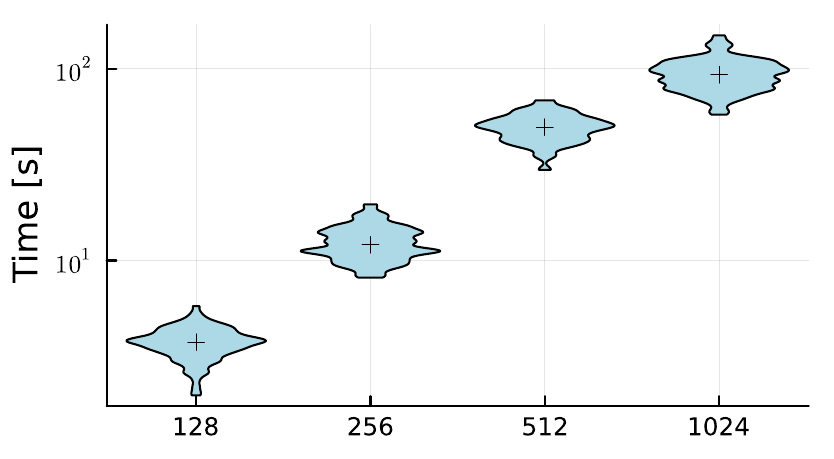}
	\includegraphics[width = .485\textwidth]{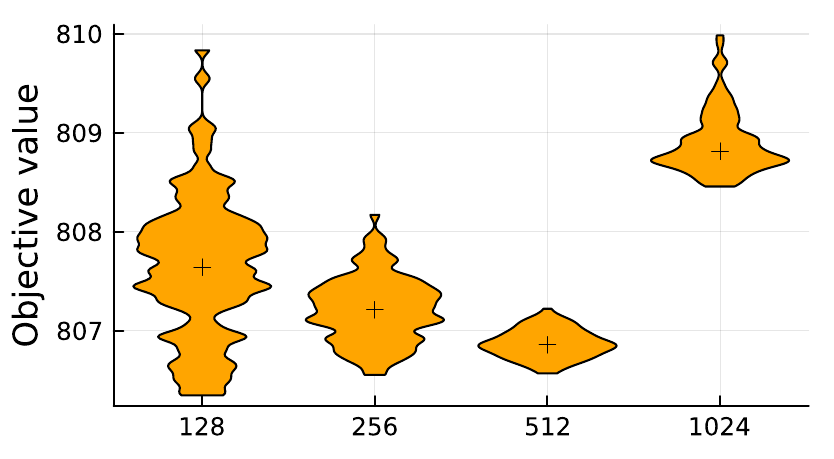}
	\vspace{.1cm}
	\begin{center}
		\begin{tikzpicture}
			\rotatebox{90}{
			\node at (0,0) {\footnotesize$\norm{(\mu_i^\top\nabla F(u)(t_i))_{i\in[n_t-1]}}_2$};}
			\node at (3.75,0) {\includegraphics[width = .485\textwidth]{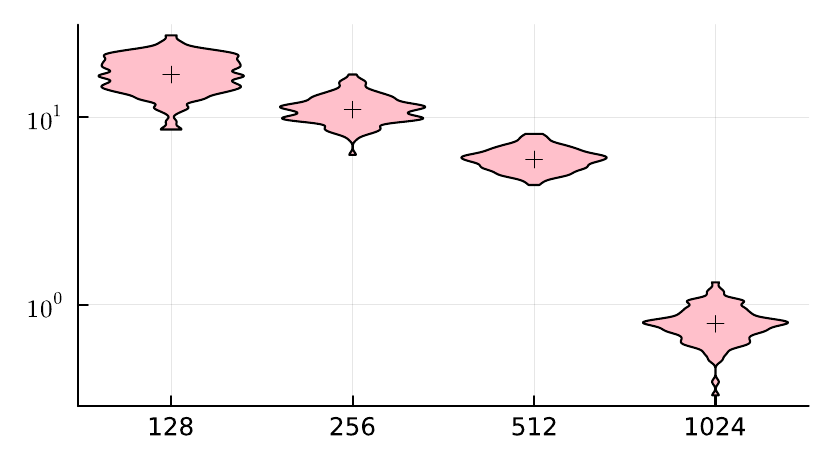}};
		\end{tikzpicture}
	\end{center}
	
	\caption{Distributions of the computation times, objective values and $\norm{(\mu_i^\top\nabla F(u)(t_i))_{i\in[n_t-1]}}_2$ for the trust region algorithm applied to \eqref{eq:Heat} for different grid sizes $2^k$, $k\in\{7,\dots,10\}$ and 500 samples. Cross represents the median.}
	\label{fig:Violin_HEAT}
\end{figure}

The results show similar patterns as in the previous example, with the notable exception that objective values for $n=1024$ are generally higher than for smaller values of $n$. However, this is due to discretization errors being more significant in the computation of the objective for rough grids. Note that the comparatively high computation time stems largely from the calculation of objective and derivative values, cf. \cref{tab:times_Heat}.

\begin{table}[htp]
	\centering
	\begin{tabular}{r|c|c|c|c}
		$n$  & 128 & 256 &  512 & 1024 \\ \hline
		range [\%] & $[49.3,59.8]$     & $[53.1,62.2]$    & $[55.2,62.9]$  & $[37.3,46.7]$ \\
		average [\%] & $53.2$     & $56.4$ & $58.1$ & $40.9$
	\end{tabular}
	\caption{Percentage of computation time for the objective and derivative w.r.t. the total computation time for solving \eqref{eq:Heat} - range and average.}
	\label{tab:times_Heat}
\end{table}

In \cref{fig:sol_HEAT}, we can see that the controls cycle between the values 2 and 3 to maintain the present temperature distribution, while at the end of the time horizon, the term $\sum u_i(t)$ in the objective leads to the control approaching zero, as the reduction of $\norm{y-y_d}_{L^2(\Omega)}^2$ is becoming less beneficial.

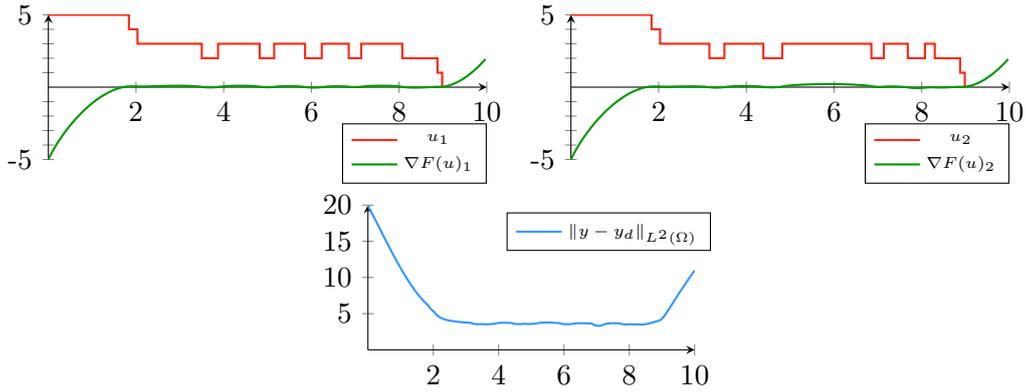
\begin{figure}[tp]
\centering
\begin{tikzpicture}
	\begin{axis}[
		xmin=0, xmax=10,
		width = .5\textwidth,
		height = 3.5cm,
		axis lines=middle,
		ytick={-5, -4, -3, -2, -1, 0, 1, 2, 3, 4, 5},
		yticklabels={-5, , , , , 0, , , , , 5},
		legend style={at={(1.,.25)}}
		]
		\addplot+[const plot, mark=none, BTUred,thick] file[skip first] {data_files/u_1__heat_switch.dat};
		\addplot+[BTUgreen, mark = none,thick] file[skip first] {data_files/nabla_f_1__heat_new2.dat};
		\legend{\tiny $u_1$,\tiny $\nabla F(u)_1$}
	\end{axis}
\end{tikzpicture}
\begin{tikzpicture}
	\begin{axis}[
		xmin=0, xmax=10,
		width = .5\textwidth,
		height = 3.5cm,
		axis lines=middle,
		ytick={-5, -4, -3, -2, -1, 0, 1, 2, 3, 4, 5},
		yticklabels={-5, , , , , 0, , , , , 5},
		legend style={at={(1.,.25)}}
		]
		\addplot+[const plot, mark=none, BTUred,thick] file[skip first] {data_files/u_2__heat_switch.dat};
		\addplot+[BTUgreen, mark = none,thick] file[skip first] {data_files/nabla_f_2__heat_new2.dat};
		\legend{\tiny $u_2$,\tiny $\nabla F(u)_2$}
	\end{axis}
\end{tikzpicture}
\begin{tikzpicture}
	\begin{axis}[
		xmin=0, xmax=10,
		ymin = 0,
		width = .4\textwidth,
		height = 3.5cm,
		axis lines=middle,
		legend style={at={(1.05,.95)}}
		]
		\addplot+[BTUlightblue, mark=none,thick] file[skip first] {data_files/y_diff.dat};
		\legend{\tiny $\norm{y-y_d}_{L^2(\Omega)}$}
	\end{axis}
\end{tikzpicture}

\caption{Solution $u$ of \eqref{eq:Heat} with scaled derivatives such that $\norm{\nabla F(u)}_{L^\infty} = 5$ and $\norm{y(t,\cdot)-y_d(t,\cdot)}_{L^2(\Omega)}$ for the corresponding state $y$ with $\beta = 10^{-1}$. Objective value $808.5$.}
\label{fig:sol_HEAT}
\end{figure}

To observe the influence of $p$ on the solution, consider \cref{fig:different_p_values_Heat}, where solutions generated by the trust region algorithm for $p\in\{1,2,4,\infty\}$ with $n=1024$ and the same initial function $u_0$ are displayed.
For $p=1$, the switching behavior does not exhibit any particular structural pattern and switches occur independently across the different components of the solution. Indeed, as $p$ increases from $p=1$ to $p\in\{2,4,\infty\}$, multiple components of the solution tend to switch simultaneously. 

\begin{figure}[htp]
	\centering
	% First row of 3 figures
	\begin{minipage}{0.48\textwidth}
		\centering
		\begin{tikzpicture}
			\begin{axis}[
				xmin=0, xmax=10,
				width = .9\textwidth,
				height = 3.2cm,
				axis lines=middle,
				ytick = {0,1,2,3,4,5},
				xticklabel = \empty,
				%legend pos=south east,
				]
				\addplot+[const plot, mark=none, BTUlightblue,thick] file[skip first] {data_files/u_1__p1_switch_heat.dat};
				\addplot+[const plot,BTUorange, mark = none,thick, dashed] file[skip first] {data_files/u_2__p1_switch_heat.dat};
			\end{axis}
		\end{tikzpicture}
		\caption*{$p=1$}
	\end{minipage}
	\begin{minipage}{0.48\textwidth}
		\centering
		\begin{tikzpicture}
			\begin{axis}[
				xmin=0, xmax=10,
				width = .9\textwidth,
				height = 3.2cm,
				axis lines=middle,
				ytick = {0,1,2,3,4,5},
				xticklabel = \empty,
				%legend pos=south east,
				]
				\addplot+[const plot, mark=none, BTUlightblue,thick] file[skip first] {data_files/u_1__p2_switch_heat.dat};
				\addplot+[const plot,BTUorange, mark = none,thick, dashed] file[skip first] {data_files/u_2__p2_switch_heat.dat};
			\end{axis}
		\end{tikzpicture}
		\caption*{$p=2$}
	\end{minipage}
	
	\vspace{-0.1em} % Space between the rows
	
	% Second row of 2 figures
	\begin{minipage}{0.48\textwidth}
		\centering
		\begin{tikzpicture}
			\begin{axis}[
				xmin=0, xmax=10,
				width = .9\textwidth,
				height = 3.2cm,
				axis lines=middle,
				ytick = {0,1,2,3,4,5},
				xticklabel = \empty,
				%legend pos=south east,
				]
				\addplot+[const plot, mark=none, BTUlightblue,thick] file[skip first] {data_files/u_1__p4_switch_heat.dat};
				\addplot+[const plot,BTUorange, mark = none,thick, dashed] file[skip first] {data_files/u_2__p4_switch_heat.dat};
			\end{axis}
		\end{tikzpicture}
		\caption*{$p=4$}
	\end{minipage}
	\begin{minipage}{0.48\textwidth}
		\centering
		\begin{tikzpicture}
			\begin{axis}[
				xmin=0, xmax=10,
				width = .9\textwidth,
				height = 3.2cm,
				axis lines=middle,
				ytick = {0,1,2,3,4,5},
				xticklabel = \empty,
				%legend pos=south east,
				]
				\addplot+[const plot, mark=none, BTUlightblue,thick] file[skip first] {data_files/u_1__pInf_switch_heat.dat};
				\addplot+[const plot,BTUorange, mark = none,thick, dashed] file[skip first] {data_files/u_2__pInf_switch_heat.dat};
			\end{axis}
		\end{tikzpicture}
		\caption*{$p=\infty$}
	\end{minipage}
	\vspace{-.1em}
	\begin{minipage}{0.5\textwidth}
		\centering
		\begin{tikzpicture}
			\draw[BTUlightblue, thick] (0,0.2) -- (0.5,0.2);
			\node at (0.8,0.2) {$u_1$};
			\draw[BTUorange, dashed, thick] (1.3,0.2) -- (1.8,0.2);
			\node at (2.1,0.2) {$u_2$};
		\end{tikzpicture}
	\end{minipage}
	
	\caption{Solutions $u$ for different values of $p$ in $\TVp$ with $n=1024$, $\beta = 10^{-2}$.}
	\label{fig:different_p_values_Heat}
\end{figure}

	\section{Conclusion and Outlook}
	\label{sec:conclusion}
	We investigated vector-valued integer control optimization problems with total variation regularization. Our research led to the introduction of a novel total variation functional, incorporating the $p$-vector norm to exert more influence on the control structure. We established first and second-order optimality conditions using a switching point reformulation by generalizing existing results of \cite{MarkoWachsmuth2022:1}.
	
	We mention that replacing $\TVp(u)$ by $\sum_{t\in S}\norm{u(t+)-u(t-)}_p+\varepsilon$ for some small $\varepsilon>0$, where $S$ is the set of all discontinuities of (some suitable representant of) $u$, would have sufficed in order to get \cref{thm:equivalence} even for $p\in\{1,\infty\}$. However, such an approach leads to unwieldy notation and we would not have been able to use existing results of the total variation functional.
	
	Using Bellman's principle akin to \cite{MarkoWachsmuth2022:1}, solutions of \eqref{eq:prob} can again be efficiently generated. We verified the ability of the trust region algorithm to yield solutions conforming to our first-order optimality conditions, or at least closely approximating them. Through numerical experiments, we demonstrated the algorithm's capability to produce satisfactory results within reasonable computation times.
	
	In the broader landscape of applications, IOCPs with state constraints are of interest. We tested the trust-region method in combination with a penalty approach. However, the results were unsatisfactory. Thus, it remains a challenge to incorporate state constraints in IOCPs with total variation regularization. Additionally, we will be looking at mixed-integer optimal control problems in future studies.
		
	\appendix

	\section{Regularity Assumptions for \texorpdfstring{\eqref{eq:Heat}}{the Heat Distribution Problem}}
	\label{sec:Appendix2}
	We consider the state equation from the control problem \eqref{eq:Heat}.
	Since we have to prove Gâteaux differentiability w.r.t.\ controls from $u \in L^1(0,T)^2$,
	we show that the solution mapping $u \mapsto y$
	of the heat equation
	is continuous from
	$L^1(0,T)^2$
	to
	$L^2( (0,T) \times \Omega)$.
	For this,
	we employ a very weak formulation.
	For a given control $u \in L^1(0,T)^2$,
	we define the function $v \in L^1(0,T; L^2(\Omega))$ via
	\begin{equation*}
		v(t,x) := \sum_{i = 1}^2 f_i(x) u_i(t)
		.
	\end{equation*}
	We test the heat equation in \eqref{eq:Heat}
	with a suitable test function $z$
	and integrate by parts.
	This leads to
	\begin{align*}
		&
		-\int_0^T \int_\Omega y \parens{ \partial_t z + \alpha \Delta z }\,\d x \d t
		=
		\int_\Omega y_0 z(0, \cdot)\, \d x
		+
		\kappa \alpha \int_0^T \int_\Gamma y_{\mathrm{out}} z\, \d x \d t
		+
		\int_0^T \int_\Omega v z \,\d x \d t
		\\
		&
		\quad
		\forall z \in Z\colon\,
		z(T,\cdot) = 0 \text{ in } \Omega,
		\;
		\nabla_x z \cdot \mathrm{n} + \kappa z = 0 \text{ on } (0,T) \times \Gamma
	\end{align*}
	which is the very weak formulation of the heat equation.
	Here,
	$Z := H^1(0,T; L^2(\Omega)) \cap L^2(0,T; H^2(\Omega))$
	and
	$\partial_t$ denotes the weak derivative w.r.t.\ the time variable.
	Note that the integral involving $v z$ is well defined,
	since
	$H^1(0,T; L^2(\Omega))$
	embeds into
	$C([0,T]; L^2(\Omega))$.

	In what follows, we argue that
	there exists a unique very weak solution
	and that the mapping
	$L^1(0,T)^2 \ni u \mapsto y \in L^2((0,T) \times \Omega )$
	is continuous.
	Then, it is straightforward to check that the reduced formulation of \eqref{eq:Heat}
	is of form
	\eqref{eq:prob}
	and
	satisfies the regularity assumptions.
	% For Differentiability, use linearity!

	To this end,
	we consider the adjoint equation
	\begin{align*}
		-\partial_t z - \alpha \Delta z &= g \quad\text{in } (0,T) \times \Omega \\
		\nabla_x z \cdot n + \kappa z &= 0 \quad\text{on } (0,T) \times \Gamma \\
		z(T, \cdot) &= 0 \quad\text{in } \Omega
	\end{align*}
	with $g \in L^2( (0,T) \times \Omega )$.
	Note that this is a heat equation running backward in time.
	One can check that this equation satisfies maximum parabolic regularity,
	see
	\cite[Theorem~5.16]{Haller-DintelmannRehberg2009}
	in combination with standard elliptic regularity results
	or
	\cite[Theorem~5.3]{ChillMeinlschmidtRehberg2020}
	in combination with \cite[Theorem~4.1]{Dore1993}.
	Consequently,
	the solution operator $S\colon g \mapsto z$ of the adjoint equation
	is continuous from
	$L^2( (0,T) \times \Omega)$
	to
	$Z$.

	We start by checking uniqueness.
	Given two very weak solutions $y_1, y_2 \in L^2( (0,T) \times \Omega )$
	we use $z := S(y_1 - y_2) \in Z$
	as a test function.
	By subtracting the two very weak formulations,
	we immediately get
	$\norm{y_1 - y_2}_{L^2( (0,T) \times \Omega )}^2 = 0$
	and this shows uniqueness.

	For the existence of a solution,
	we employ an adjoint approach.
	Let us denote by
	$R\colon L^2( (0,T) \times \Omega ) \to C([0,T];L^2(\Omega)) \times L^2( (0,T) \times \Gamma) \times L^2(\Omega)$
	the mapping
	\begin{equation*}
		g \mapsto
		R g
		:=
		(
			S g,
			(S g)|_{(0,T) \times \Gamma},
			(S g)(0, \cdot)
		)
		.
	\end{equation*}
	Consequently,
	the adjoint maps from
	$M([0,T];L^2(\Omega)) \times L^2( (0,T) \times \Gamma) \times L^2(\Omega)$
	to
	$L^2( (0,T) \times \Omega )$.
	For given $u$, we define $v$ as above and set
	$y := R\adjoint ( v, y_{\mathrm{out}}, y_0 )$.
	In order to check that $y$ satisfies the very weak solution,
let $z \in Z$ with
	$z(T,\cdot) = 0$ in $\Omega$
	and
	$\nabla_x z \cdot \mathrm{n} + \kappa z = 0$ on $(0,T) \times \Gamma$
	be given.
	Note that $z = S g$
	with $g := -\partial_t z - \alpha \Delta z$.
	By definition of the adjoint operator,
	we obtain
	\begin{align*}
		-\int_0^T \int_\Omega y \parens{ \partial_t z + \alpha \Delta z } \, \d x \d t
		&=
		\dual{y}{ -\partial_t z - \alpha \Delta z }
		=
		\dual{R\adjoint ( v, y_{\mathrm{out}}, y_0 )}{ g }
		\\&
		=
		\dual{( v, y_{\mathrm{out}}, y_0 )}{R g}
		=
		\dual{( v, y_{\mathrm{out}}, y_0 )}{
		(
			z,
			z |_{(0,T) \times \Gamma},
			z(0, \cdot)
		)
		}
		\\&
		=
		\int_\Omega y_0 z(0, \cdot) \, \d x
		+
		\kappa \alpha \int_0^T \int_\Gamma y_{\mathrm{out}} z \, \d x \d t
		+
		\int_0^T \int_\Omega v z \, \d x \d t
		.
	\end{align*}
	This shows that $y$ is indeed a very weak solution.
	Due to continuity of the adjoint $R\adjoint$,
	we also get that the mapping from
	$u \in L^1(0,T)^2$
	to
	$y \in L^2( (0,T) \times \Omega)$
	is continuous.

	%%fakesection: Bib
	\printbibliography
	
\end{document}